\documentclass[a4paper,reqno, 11pt]{amsart}   	
\usepackage[DIV=12, oneside]{typearea}			
\usepackage[utf8]{inputenc}
\usepackage[T1]{fontenc}
\usepackage[english]{babel}
\usepackage[centertags]{amsmath}
\usepackage{amstext,amssymb,amsopn,amsthm}
\usepackage{mathrsfs}
\usepackage{dsfont}
\usepackage{bbm}
\usepackage{thmtools}
\usepackage{mathtools}
\usepackage{graphicx}
\usepackage[titletoc,title]{appendix}
\usepackage{etexcmds}

\usepackage{multirow}

\usepackage{tikz}
\usepackage{pgfplots}
\pgfplotsset{compat=newest}
\usetikzlibrary{arrows.meta, decorations.pathreplacing, positioning}
\usetikzlibrary{calc, patterns,angles,quotes, tikzmark, fit, shapes.geometric}
\tikzset{%
	ebo unit/.store in=\ebounit,
	ebo corners/.style={rounded corners=#1\ebounit},
}
\usepgfplotslibrary{fillbetween}
\definecolor{c595959}{RGB}{89,89,89}
\definecolor{c5a5a5a}{RGB}{90,90,90}

\usepackage{enumitem}
\setlist[enumerate]{itemsep=0mm}

\usepackage[colorlinks=true, linkcolor=black, urlcolor={black},
pdfborder={0 0 0}, citecolor=black]{hyperref}

\renewcommand{\d}{\textnormal{d}}

\addto\extrasenglish{}
\addto\extrasenglish{}
\addto\extrasenglish{}

\theoremstyle{plain}
\declaretheorem[title=Theorem, parent=section]{theorem}
\declaretheorem[title=Lemma,sibling=theorem]{lemma}
\declaretheorem[title=Proposition,sibling=theorem]{proposition}

\theoremstyle{definition}
\declaretheorem[title=Definition,sibling=theorem]{definition}
\declaretheorem[title=Remark,sibling=theorem]{remark}
\declaretheorem[title=Remark, numbered=no]{remark*}

\declaretheorem[title=Assumption, numbered=no]{assumption*}

\numberwithin{equation}{section}

\newcommand{\N}{\mathds{N}}
\newcommand{\R}{\mathds{R}}
\newcommand{\C}{\mathds{C}}
\newcommand{\Z}{\mathds{Z}}

\newcommand{\sS}{\mathscr{S}}

\newcommand{\sF}{\mathscr{F}}

\def\hmath$#1${\texorpdfstring{{\rmfamily\textit{#1}}}{#1}}

\newcommand{\eps}{\varepsilon}

\newcommand{\loc}{\mathrm{loc}}

\newcommand{\I}[1]{(\text{I}_{#1})}
\newcommand{\II} [1] {(\text{II}_{#1})}
\newcommand{\III}[1]{(\text{III}_{#1})}
\newcommand{\IV}[1]{(\text{IV}_{#1})}

\DeclareMathOperator{\dist}{dist}

\DeclareMathOperator{\diam}{diam}

\DeclareMathOperator{\supp}{supp}

\DeclareMathOperator{\pv}{p.v.}

\newcommand{\1}{{\mathbbm{1}}}

\usepackage{esint}
\newcommand{\norm}[1]{\left\lVert#1\right\rVert} 
\newcommand{\abs}[1]{\ensuremath{\left\vert#1\right\vert}} 
 
\newcommand{\distw}[2]{\dist(#1, #2)} 

\newcommand{\Gs}{\mathcal{G}_s(\lambda,\Lambda)}

\begin{document}
	\allowdisplaybreaks
	\title[Boundary regularity for general elliptic operators of order $2s$]{Boundary regularity for general elliptic operators \\ of order $2s$} 
	
	\author{Florian Grube}
	\author{Xavier Ros-Oton}

	\address{Fakult{\"a}t f{\"u}r Mathematik, Universit{\"a}t Bielefeld, Postfach 10 01 31, 33501 Bielefeld, Germany}
	\email{fgrube@math.uni-bielefeld.de}
	\address{ICREA, Pg. Llu\'is Companys 23, 08010 Barcelona, Spain \& Universitat de Barcelona, Departament de Matem\`atiques i Inform\`atica, Gran Via de les Corts Catalanes 585, 08007 Barcelona, Spain \& Centre de Recerca Matem\`atica, Barcelona, Spain}
	\email{xros@icrea.cat}
	\makeatletter
	\@namedef{subjclassname@2020}{%
		\textup{2020} Mathematics Subject Classification}
	\makeatother
	
	

	\begin{abstract} 
We establish optimal $C^s$ boundary regularity for the most general class of (linear and translation invariant) nonlocal elliptic operator of order $2s$.
Namely, we consider L\'evy operators that are symmetric and its Fourier symbol satisfies $\mathcal{A}(\xi)\asymp |\xi|^{2s}$ in $\R^d$. 
This was only known when the kernel of the operator (or L\'evy measure) is either homogeneous or comparable to that of the fractional Laplacian, with different proofs in each case.
Our new proofs extend both at the same time, and work in a very general class of domains, under a $C^1$-Dini-type condition.
	\end{abstract}

	\hypersetup{pageanchor=false}
	\maketitle
	\hypersetup{pageanchor=true}

\section{Introduction}\label{sec:intro}

Nonlocal operators of order $2s$, with $s\in(0,1)$, of the type 
\begin{equation*}
	Lu(x):= \int_{\R^d}\big( u(x)-u(x+y)\big)K(y)dy,
\end{equation*}
with $K\geq0$ and $K(y)=K(-y)$, have attracted great interest in the last decades, especially since the works of Caffarelli and Silvestre on this topic \cite{Bas09,BaLe02,BaLe02b,BKK08,CaSi09,CaSi11a,CaSi11b,CKS10,ChSo98,DyKa20,Kas09,KaWe23,RS16a,Sil06}.
These operators appear naturally in a variety of settings, most notably in the study of L\'evy processes in Probability, the Boltzmann equation in Mathematical Physics, or pseudodifferential operators in Analysis; see e.g. \cite{FeRo24}.

A key question in this context is that of regularity: assume we have $Lu=f$ in $\Omega\subset\R^d$, what can we then say about the smoothness of solutions $u$?

The \emph{interior} regularity of solutions to these equations is nowadays very well understood. 
It was first studied for the fractional Laplacian and for kernels satisfying 
\begin{equation}\label{X1}
\frac{\lambda}{|y|^{d+2s}}\leq K(y) \leq \frac{\Lambda}{|y|^{d+2s}},\qquad 0<\lambda\leq \Lambda,
\end{equation}
and more recently for general elliptic operators of order $2s$, i.e., those for which
\begin{equation}\label{X2a}
\int_{B_{2r}\setminus B_r} K(dy) \leq \Lambda r^{-2s}\qquad \forall r>0,
\end{equation}
\begin{equation}\label{X2b}
\qquad\qquad\int_{B_r} |y\cdot e|^2 K(dy) \geq \lambda r^{2-2s}\qquad \forall r>0,\ e\in\mathbb{S}^{d-1}.
\end{equation}
Here, the kernel $K$ is not necessarily absolutely continuous, but can be a measure.

These two conditions are equivalent to the Fourier symbol $\mathcal{A}$ of $L$ satisfying
\[\tilde\lambda |\xi|^{2s} \leq \mathcal{A}(\xi) \leq \tilde \Lambda |\xi|^{2s},\qquad 0<\tilde\lambda\leq \tilde\Lambda,\]
and thus this is the most general class of elliptic operators of order $2s$ for which one can expect maximum principle and regularization properties; see e.g. \cite[Chapter 2]{FeRo24}.
We denote this class $\Gs$; see \autoref{sec:prelims} below.

The \emph{boundary} regularity of solutions was first studied for the fractional Laplacian \cite{Get61,Lan72,RS14}, and later for general operators $L$ that in addition satisfy
\begin{equation}\label{X3}
K\ \textrm{is homogeneous of degree}\ -d-2s;
\end{equation}
see \cite{AbGr23,Gru15,Gru22,RS16,RS16a, Gru24}. These are usually called stable operators.

It turns out that, while solutions are $C^{2s+\alpha}_{\rm loc}(\Omega)$ in the interior (provided that $f$ is H\"older continuous), the optimal regularity of solutions up to boundary is $C^s(\overline\Omega)$.
Proving this for general stable operators \eqref{X3} relies very strongly on the fact that $(x_n)_+^s$ is an exact solution in $\{x_n>0\}$, for \emph{any} operator in this class; see \cite[Lemma 2.6.2]{FeRo24}.

Quite recently, the second author and Weidner proved that, quite surprisingly, $C^s$ regularity still holds for \emph{all} operators $L$ satisfying \eqref{X1}.
The proof of this is not based at all on explicit barriers, but on a construction through a free boundary problem that strongly relies on interior and boundary Harnack inequalities; see \cite{RoWe24} for more details.

The question we tackle in this paper is:
\[\textit{Does $C^s$ boundary regularity hold for general elliptic operators \eqref{X2a}-\eqref{X2b} of order $2s$?}\]

Notice that the proofs for operators satisfying \eqref{X3} rely on the explicit barrier $(x_+)^s$, while for those satisfying \eqref{X1} rely on Harnack inequality. Both fail for general operators \eqref{X2a}-\eqref{X2b}.

\subsection{Main results}

Here, we answer the question positively for general operators of order $2s$ with new techniques, using ideas from Fourier analysis to establish the following 1D Liouville theorem. 

\begin{theorem}\label{thm-intro1}
Let $L\in \Gs$.
Then, there exists a continuous function $b$ satisfying 
\[Lb=0 \quad\textrm{in}\quad \R_+, \qquad b=0\quad \textrm{in}\quad \R_-,\]
which is unique up to multiplicative constant among all functions with growth 
\[|b(t)|\leq C(1+|t|^{2s-\varepsilon}) \quad \textrm{in}\quad \R_+,\]
for some $\varepsilon>0$.
Furthermore, $b\in C^s_{\rm loc}(\R)$ and 
\[c(t_+)^s \leq b(t) \leq C(t_+)^s\quad \textrm{in}\quad \R\]
for some positive constants $C$ and $c$ depending only on $s,\lambda,\Lambda$.
\end{theorem}

Thanks to this, we prove $C^s$ regularity in $C^{1,\omega}$ domains, as well as a Hopf lemma and a boundary Harnack principle.
We assume the modulus of continuity $\omega$ satisfies 
		\begin{equation}\label{eq:s-dini}
			\int_0^1 \frac{\omega(t)^s}{t}\d t <\infty.
		\end{equation}

In particular, our results apply to $C^{1,\alpha}$ domains, and our new proofs unify those for \eqref{X3} and for \eqref{X1} at the same time, while at the same time applies to a more general class of domains than in \cite{RS14,RS16,RoWe24}.

\begin{theorem}\label{thm-intro2}
Let $L\in \Gs$ and $\Omega\subset \R^d$ be any bounded $C^{1,\omega}$ domain, with $\omega$ satisfying \eqref{eq:s-dini}.
Let $f\in L^\infty(\Omega\cap B_2)$, and  $u\in L_{2s-\eps}^\infty(\R^d)$ be any weak solution of
		\begin{equation}\label{eq:localised-equation}
			\begin{split}
				L u &= f \text{ in }\Omega\cap B_2,\\
				u&=0 \text{ on }\Omega^c\cap B_2.
			\end{split}
		\end{equation}
Then, $u\in C^s({B_{1}})$ with
		\begin{equation*}
			\norm{u}_{C^s(B_1)}\le C\big(\norm{f}_{L^\infty(\Omega\cap B_2)}+ \norm{u}_{L^\infty_{2s-\eps}(\R^d)} \big).
		\end{equation*}
		The constant $C$ depends only on $s,\eps,\lambda,\Lambda$, and $\Omega$. 
		
Furthermore, if $f\geq0$ in $\Omega\cap B_2$ and $u\geq0$ in $B_2^c$, then
 \[u\geq cd_\Omega^s \quad\textrm{in}\quad \Omega\cap B_1\]
for some $c>0$, where $d_\Omega$ is the distance to the boundary.
\end{theorem}

The boundary Harnack principle reads as follows.

\begin{theorem}\label{th:boundary-harnack}
Let $L\in \Gs$ and $\Omega\subset \R^d$ be any bounded $C^{1,\omega}$ domain, with $\omega$ satisfying \eqref{eq:s-dini}.
Let $f_1,f_2\in L^\infty(\Omega\cap B_2)$ and  $u_1,u_2\in L_{2s-\eps}^\infty(\R^d)$ be weak solutions of
		\begin{equation}\label{eq:localised-equation-i}
			\begin{split}
				L u_i &= f_i \text{ in }\Omega\cap B_2,\\
				u_i&=0 \text{ on }\Omega^c\cap B_2,
			\end{split}
		\end{equation}
with $u_2>0$ in $\Omega\cap B_2$, $u_2\ge 0$ on $\R^d$, and $f_2\geq0$ in $\Omega\cap B_2$.
Then, $u_1/u_2\in C^\alpha(\overline\Omega\cap B_1)$ for any $\alpha\in(0,s)$.
\end{theorem}

As said above, these results were only known when $L$ satisfies either \eqref{X3} or \eqref{X1}, by completely different arguments.

Our new results hold for the most general class of (linear and symmetric) elliptic operators of order $2s$.
Moreover, $C^s$ boundary regularity is known to fail for (linear) non-symmetric operators \cite{DRSV22} as well as for (symmetric) fully nonlinear operators \cite{RS16a}.
$C^s$ boundary regularity is also known to fail for nonlocal equations with bounded measurable coefficients, both in non-divergence or in divergence form \cite{RS16a}.
Let us also refer to \cite{ChSo26} for optimal domain assumptions under which solutions are $C^s$ for operators $L$ satisfying \eqref{X1}.

\subsection{Strategy of the proofs}

Let us describe the main ideas and ingredients of our proofs.

	\subsubsection{1D Liouville theorem} 
A first key contribution of our work is the 1D Liouville theorem.
Both existence and uniqueness were open, and are needed in the proofs of Theorems \ref{thm-intro2} and~\ref{th:boundary-harnack}. 

To establish the existence of a solution $b$ we employ results from the field of L{\'e}vy-processes; see \cite{GrRy12} and \cite{BGR15}. 
Therein, the behavior of the expected exit time of the associated stochastic process from an interval $(0,R)$ is studied. 
Using known interior regularity results and taking appropriate limits $(R\to \infty)$ allows us to construct half line solutions that grow like~$t_+^s$.

The most difficult part here turns out to be the uniqueness.
We do not have any tools like Harnack inequality, nor we have any qualitative properties of the solution $b$ that we can use.
Instead, we find a completely new argument based on Fourier analysis and the so-called Wiener-Hopf factorization.
The argument is as follows.

First, we consider the Fourier transform $\hat b(\xi)$ which (by Paley-Wiener theory, and since $b\equiv0$ in $\R_-$) is analytic in the lower half-plane ${\rm Im}(\xi)<0$.
On the other hand, by the same reason, since $Lb\equiv0$ in $\R_+$, $\hat f(\xi):=\widehat{Lb}(\xi)$ is analytic in the upper half-plane  ${\rm Im}(\xi)>0$.
Moreover, on the real line we have $\hat f=\mathcal{A}\hat{b}$, where $\mathcal{A}$ is the symbol of $L$.

Using the properties of $L$, we can show that its symbol $\mathcal{A}$ can be factored as 
\[\mathcal{A}=\mathcal{A}_+\mathcal{A}_-,\]
where $\mathcal{A}_+$ (resp. $\mathcal{A}_-$) is analytic and has no zeroes in the upper (resp. lower) half plane.
This is called a Wiener-Hopf factorization.

Combining the previous information, we get that 
\[\mathcal{A}_-\hat b=\frac{\hat f}{\mathcal{A}_+},\]
where the LHS is analytic in  ${\rm Im}(\xi)<0$ and the RHS is analytic in  ${\rm Im}(\xi)>0$.
This means they both equal a function $J(\xi)$ which is analytic in  ${\rm Im}(\xi)\neq0$.
Furthermore, we prove that this function is analytic everywhere, except possibly at the origin.

By analyzing the growth of $J$ both at infinity and at the origin, we are able to apply Liouville theorem for analytic functions in $\mathbb C$ to deduce that we must have 
\[J(\xi)=\frac{a}{\xi} \qquad \textrm{and thus}\qquad \hat b(\xi) = \frac{a}{\xi \mathcal{A}_-(\xi)}\]
for some constant $a$.
And this means that the solution is unique up to multiplicative constant.

	\subsubsection{Boundary regularity} 
	
	In order to prove boundary regularity like \autoref{th:Cs-regularity}, it suffices to combine interior regularity, see \cite[Theorem 2.4.3, Proposition 2.4.4]{FeRo24}, with a boundary decay estimate of the type $|u(x)|\le C d_{\Omega}(x)^s$. To obtain this boundary decay estimate, we use a barrier function in the half space that we get from Theorem \ref{thm-intro1}. 
	
	The proof in \cite{RoWe24} (as well as previous ones) is based on a contradiction compactness arguments that works in $C^{1,\alpha}$ domains.
We expect the same method to be applicable in our situation, however, since we want to treat more general domains, we need a different approach because the contradiction compactness argument seems incompatible with domains more general than $C^{1,\alpha}$.

We circumvent this issue by employing a dyadic linear approximation of the boundary at a point $z\in \partial\Omega$, say $z=0$, strictly from outside of the domain, see \autoref{fig:geom-iteration}. 
This allows us to inductively compare the real solution $u$ to a rotated, translated, and truncated half space solution. 
A comparison principle yields positive constants $M_k$ such that
	\begin{equation*}
		u(x)\le M_k (x_d+p_k)_+^s \text{ for }x\in B_{r_k}(0)\cap \Omega.
	\end{equation*} 
	This is an adaptation of the classical ideas from \cite{LaUr88} to the nonlocal setting. 
		
	\begin{figure}[!ht]
		\centering
		\begin{tikzpicture}[scale=1.3]
		
		\draw[-{Stealth}, gray, thick] (-4.5, 0) -- (4.5, 0) node[right, black] {$x_d = 0$};
		\draw[-{Stealth}, gray, thick] (0, -2.8) -- (0, 4.5) node[above, black] {$x_d$};

		\begin{scope}
			\clip (0,0) circle (4);
			\fill[black!20, opacity=0.3] (-4, -2.2) rectangle (4, 4);
			\draw[very thick, black!40] (-4, -2.2) -- (4, -2.2);
		\end{scope}
		\draw[thick, dashed, darkgray] (0,0) circle (4);
		\node[black, font=\large] at (-2.6, 2.6) {$B_{r_0}$};
		\node[black, right, font=\small] at (3.35, -2.2) {$x_d = -p_0$};
		
		\begin{scope}
			\clip (0,0) circle (2);
			\fill[black!35, opacity=0.4] (-2, -0.4) rectangle (2, 2);
			\draw[very thick, black!60] (-2, -0.4) -- (2, -0.4);
		\end{scope}
		\draw[thick, dashed, darkgray] (0,0) circle (2);
		\node[black, font=\large] at (-1.2, 1.2) {$B_{r_1}$};
		\node[black, right, font=\small] at (2, -0.4) {$x_d = -p_1$};
		
		\begin{scope}
			\clip (0,0) circle (1);
			\fill[black!60, opacity=0.3] (-1, -0.1) rectangle (1, 1);
			\draw[very thick, black!80] (-1, -0.1) -- (1, -0.1);
		\end{scope}
		\draw[thick, dashed, darkgray] (0,0) circle (1);
		\node[black, font=\large] at (-0.43, 0.57) {$B_{r_2}$};
		\node[black, right, font=\tiny] at (0.95, -0.15) {$x_d = -p_2$};
		
		\filldraw[black] (0,0) circle (1.5pt) node[above right] {$0$};

		\draw[decorate, decoration={brace, amplitude=4pt}, thick] 
		(-2, -0.24) -- (-2, 0) node[midway, left=4pt] {$\delta_1$};
		\draw[decorate, decoration={brace, amplitude=4pt}, thick] 
		(-4, -1.89) -- (-4, 0) node[midway, left=4pt] {$\delta_0$};
		
		\draw[very thick, black] plot[domain=-4.5:4.5, samples=100] (\x, {0.03*\x*\x*\x}) 
		node[right] {$\partial \Omega$};
		\node[black, font=\Large] at (3.65, 2.75) {$\Omega$};
		
	\end{tikzpicture}
		\caption{Dyadic linear approximation of the boundary in the proof of \autoref{th:boundary-estimate}.}
		\label{fig:geom-iteration}
	\end{figure}
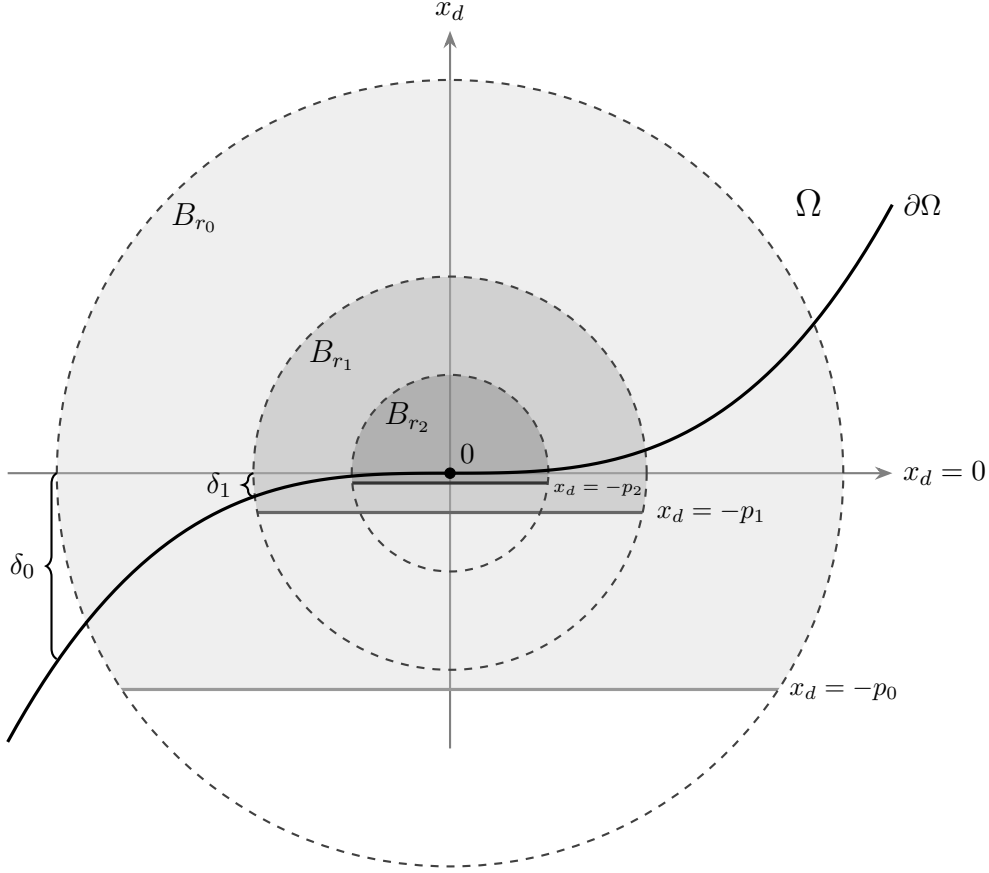

	The versatility of the dyadic approximation method extends to lower bounds $u(x)\ge Cd_{\Omega}(x)^s$. Instead of approximating the domain from outside, see \autoref{fig:geom-iteration}, the key difference is an approximation from within the domain. 
	Coupling this with a comparison with the half space solution in each dyadic ball establishes the Hopf lemma \autoref{th:hopf}, which is a critical ingredient for the boundary Harnack principle. 
		
	Finally, the boundary Harnack principle \autoref{th:boundary-harnack} follows from coupling the optimal boundary regularity \autoref{th:Cs-regularity} and the Hopf boundary lemma with the 1D Liouville theorem, which allows us to prove a higher-order expansion of the solution at a boundary point using a contradiction compactness argument, see \autoref{th:expansion}. 
	The main difficulty lies in obtaining the Liouville theorem, explained above.

	\subsection{Outline} The article is structured as follows. In \autoref{sec:prelims}, we establish the necessary mathematical framework, introduce the notation, function spaces, geometric assumptions, and solution concepts used in this article, and collect and provide some preliminary results. Our research begins with a series of results in one dimension. We construct a nontrivial solution on the half line and provide a Liouville theorem on the half line in \autoref{sec:half-line}. The proof of \autoref{th:Cs-regularity} using a dyadic linear approximation of the boundary is the focal point of \autoref{sec:boundary-estimate}. In \autoref{sec:hopf}, we provide a Hopf-type boundary lemma by adapting the methods from the previous section. The proof of \autoref{thm-intro2} is contained in the end of \autoref{sec:hopf}. Regularity estimates and the one-dimensional analogue allow us to prove a Liouville theorem in the half space in all dimensions in \autoref{sec:liouville}. The proof of the boundary Harnack principle \autoref{th:boundary-harnack} is contained in \autoref{sec:bhp}. 
	
	\subsection*{Acknowledgments} FG was financially supported by the German Research Foundation (DFG - Project number 541771122). 
	XR was supported by the European Union under the ERC Consolidator Grant No 101123223 (SSNSD), by the AEI project PID2024-156429NB-I00 (Spain), the AEI-DFG project PCI2024-155066-2 (Spain-Germany), the AEI Grant RED2024-153842-T (Spain), and the AEI Maria de Maeztu Program for Centers and Units of Excellence in R$\&$D CEX2020-001084-M.

	\section{Preliminaries}\label{sec:prelims}
	We briefly introduce the notation, conventions, boundary assumptions, and function spaces used throughout this article. 
	
	The positive part of a real number $r$ is denoted by $r_+\coloneq \max\{r,0\}$ and the negative part $r_-\coloneq (-r)_+$. The unit sphere in the $d$-dimensional Euclidean space is denoted by $S^{d-1}$. A ball with radius $r$ and with center $x\in \R^d$ is written as $B_r(x)$. At times, we omit the center point and simply write $B_r$. By $\Gamma(\cdot)$ we denote Euler's gamma function. For a vector $x\in \R^d$, we use the convention $x=(x',x_d)$ where $x'\in \R^{d-1}$ and $x_d\in \R$. We write $e_d\coloneq (0,\dots, 0, 1)\in \R^d$. By $\sF$ or $\hat{(\cdot)}$, we denote the Fourier transform. 
	
	In our proofs, we use generic constants $C$ which depend only on universal quantities. These constants may change from line to line. At times, we use numbered constants $c_1, c_2, \dots$ which are fixed quantities for each proof. 
	
	Throughout this article, we always use a bounded open set $\Omega$, for which we introduce the notation $x\mapsto d_x\coloneq \min\{ \abs{x-y}\mid y\in \partial\Omega \}$ for the distance from $x$ to the boundary of $\Omega$. \smallskip
	
	The space of Hölder continuous functions on $\overline{\Omega}$ is written as $C^{\alpha}(\overline\Omega)$, its seminorm is denoted by $[\cdot]_{C^{\alpha}(\overline{\Omega})}$. We also use $\dot{C}^\alpha(\Omega)$ to denote the homogeneous set of continuous functions $f:\Omega\to \R$ such that $[f]_{C^{\alpha}(\Omega)}<\infty$. For $\alpha>0$, we introduce the tail space $L_{\alpha}^\infty(\R^d)$ as the space of functions $f\in L^\infty_{\loc}(\R^d)$ such that the weighted norm $\|f\|_{L_{\alpha}^\infty(\R^d)}\coloneq\|f/(1+|\cdot|)^{\alpha}\|_{L^\infty(\R^d)}$ is finite.
	
	The class of operators $\Gs$ consists of all operators $L$ of the form 
	\begin{equation*}
		Lu(x):= \int_{\R^d}\big( u(x)-u(x+y)\big)K(\d y),
	\end{equation*}
	that satisfies \eqref{X2a} and \eqref{X2b}. We use the notation $\mathcal{A}$ for the Fourier-symbol of $L\in \Gs$, i.e., $\sF [L \phi](\xi) = \mathcal{A}(\xi)\hat{\phi}(\xi)$. 
	
	\begin{definition}\label{def:weak-distr-sol}
		Let $L \in \Gs$, $\Omega\subset \R^d$ be an open set, and $f\in L^2(\Omega)$. We say that $u$ is a
		\begin{enumerate}
			\item[(i)] weak (sub-)solution to \eqref{eq:localised-equation} if $\mathcal{E}(u)<\infty$ and
			\begin{equation*}
				\int_{\R^d}\int_{\R^d}\big(u(x)-u(x+h)\big)\big( v(x)-v(x+h) \big)K(\d h)\d x\substack{=\\(\le)} \int_{\Omega} f(x) v(x) \d x
			\end{equation*}
			for any $v\in L^2(\R^d)$ such that $v=0$ on $\Omega^c$ and $\mathcal{E}(v)<\infty$ where
			\begin{equation*}
				\mathcal{E}(u)\coloneq \int_{\R^d}\int_{\R^d}\1_{(\Omega^c\times \Omega^c)^c}(x,x+h)\big(u(x)-u(x+h)\big)^2K(\d h)\d x.
			\end{equation*}
			\item[(ii)] distributional (sub-)solution to \eqref{eq:localised-equation} if 
			\begin{equation*}
				\int_{\R^d}u(x)L \eta(x)\d x \substack{=\\(\le)} \int_{\Omega} f(x)\eta(x)\d x
			\end{equation*}
			for any $\eta\in C_c^\infty(\Omega)$.
			\item[(iii)] We say that a function $u$ is a continuous distributional solution to \eqref{eq:localised-equation} in a set $\Omega$ if it is a distributional solution and continuous on the closure $\overline{\Omega}$.
		\end{enumerate}
	\end{definition}

	\begin{definition}\label{def:modulus}
		We say that $\omega:[0,\infty)\to [0,\infty)$ is a modulus of continuity (or modulus) if it is increasing, concave, $\omega(0)=0$, and it has sublinear growth at infinity. 
	\end{definition}
	
	\begin{definition}\label{def:paraboloid}
		Let $\omega$ be a modulus of continuity. An open set $\Omega$ satisfies the exterior  $C^{1,\omega}$-paraboloid property at a boundary point $z\in \partial\Omega$ with radius $r_0$ if we find a rotation $R_z$ such that 
		\begin{equation*}
			R_z\Big(\{-z\}+\Omega\Big)\cap B_{r_0}(0)\subset  \{ (x',x_d)\in B_{r_0}(0)\mid x_d>-|x'|\omega(|x'|) \}.
		\end{equation*}
		An open set $\Omega$ satisfies the interior $C^{1,\omega}$-paraboloid property at a boundary point $z\in \partial\Omega$ with radius $r_0$ if we find a rotation $R_z$ such that 
		\begin{equation*}
			R_z\Big(\{-z\}+\Omega\Big)\cap B_{r_0}(0)\supset  \{ (x',x_d)\in B_{r_0}(0)\mid x_d>|x'|\omega(|x'|) \}.
		\end{equation*}
	\end{definition}
		
	\begin{lemma}\label{lem:basic-properties-Gs}
		Let $L\in \Gs$, $\theta\in S^{d-1}$, and $r>0$. Let $u:\R^d\to \R$ and $v:\R\to \R$. If $u(x)=v(x\cdot \theta)$, then 
		\begin{equation*}
			L u(x)= \tilde{L} v(x\cdot \theta)
		\end{equation*}
		for some $\tilde{L}\in \Gs$. Moreover, if $u_r(x)\coloneq   u(rx)$, then 
		\begin{equation*}
			L_ru_r(x)= r^{2s}L u (rx)
		\end{equation*}
		for $L_r\in \Gs$ where its L{\'e}vy measure is $K_r\coloneq   r^{2s}K(r (\cdot))$. 
	\end{lemma}
	\begin{proof}
		This lemma is a direct consequence of the definition of the class $\Gs$, see also \cite[Section 2.1.7, Remark 2.1.19]{FeRo24}.
	\end{proof}
	
	The next result is a direct consequence of \cite[Proposition 3.5]{GrRy12} and \cite[Proposition 2.4]{BGR15}. It is the main ingredient in the construction of a half line solution that is comparable to $t_+^s$. This allows us to construct a barrier in the half space.
	\begin{proposition}\label{prop:torsion_estimate}
		Let $L$ be an operator in the class $\Gs$. The unique weak solution $\tilde{b}_R$ to 
		\begin{equation*}
			\begin{split}
				L \tilde{b}_R&= 1 \text{ in }(0,R),\\
				\tilde{b}_R&=0 \text{ on }(0,R)^c
			\end{split}
		\end{equation*}
		satisfies 
		\begin{equation}\label{eq:torsion_estimate}
			C^{-1}\lambda R^s (x\wedge (R-x))^s\le \tilde{b}_R(x)\le C\Lambda R^s (x\wedge (R-x))^s
		\end{equation}
		where the constant $C$ depends only on $s$.
	\end{proposition}
	\begin{proof}
		Let $X_t$ be the L{\'e}vy process with generator $-L$. Note that 
		\begin{equation*}
			\int_{\R} \left(1\wedge \frac{\abs{h}^2}{R^2}\right)K(\d h)\ge R^{-2}\int_{B_R} \abs{h}^2K(\d h)\ge \lambda R^{-2s}
		\end{equation*} 
		and 
		\begin{align*}
				\int_{\R} \left(1\wedge \frac{\abs{h}^2}{R^2}\right)K(\d h)&\le \sum_{k\in \N} \Big(2^{2-2k}K( B_{2^{1-k}R}\setminus B_{2^{-k}R} ) + K(B_{2^{k}R}\setminus B_{2^{k-1}R})\Big)\\
				&\le 4\Lambda R^{-2s}\sum_{k\in \N} \big(2^{-2k(1-s)} + 2^{-2sk}\big).
		\end{align*}
		By \cite[Proposition 3.5]{GrRy12} and \cite[Proposition 2.4]{BGR15}, the function $\mathbb{E}^x[\tau_{(0,R)}]$ satisfies the bound \eqref{eq:torsion_estimate}. Here, the random variable $\tau_A$ is the first exit time of the process $X_t$ from the set $A\subset \R$, i.e., $\tau_{A}\coloneq   \inf\{ t>0\mid X_t\notin A \}$. \smallskip
		
		\textit{Claim 1.} The function $v(x)\coloneq  \mathbb{E}^x[\tau_{(0,R)}]$ is a distributional solution to $L v=1$ in $(0,R)$ and $v=0$ on $(0,R)^c$. 
		
		Due to the previous bounds, we know that $v\in L^\infty((0,R))$. Moreover, the exterior condition, i.e., $v=0$ on $\Omega^c$, is met since $X_0=x$ by definition and, thus, $\tau_{(0,R)}^x=0$ if $x\notin (0,R)$. 
		
		The result follows from standard potential theory for Markov processes. We provide the proof for the convenience of the reader. For the moment, we let $\alpha>0$ be arbitrary. The resolvent $G_\alpha$ is defined as 
		\begin{equation*}
			G_\alpha \phi(x)\coloneq  \mathbb{E}^x\left[\int_{0}^{\infty } e^{-\alpha t }\phi(X_t)\d t \right].
		\end{equation*}		
		We define an approximation of $v$ using the resolvent 
		\begin{equation*}
			v_\alpha (x)\coloneq   \mathbb{E}^x\left[\int_{0}^{\tau_{(0,R)} } e^{-\alpha t }\d t \right]= G_\alpha[\1_{(0,R)}](x).
		\end{equation*}
		This definition immediately yields $v_\alpha(x) \le v(x)$ and $v_\alpha(x)\to v(x)$ monotonously as $\alpha \to 0+$. Since $v$ is bounded and has compact support, so is $v_\alpha$. By Dynkin's formula, see \cite{Dyn65}, we find for any $\phi\in C_c^\infty((0,R))$
		\begin{equation}\label{eq:dynkin}
			\begin{split}
				\mathbb{E}^x[e^{-\alpha \tau_{(0,R)}}\phi(X_{\tau_{(0,R)}})]&=\phi(x)+ \mathbb{E}^x\left[ \int_{0}^{\tau_{(0,R)}} (\partial_t -L)e^{-\alpha t}\phi(X_t)\d t \right]\\
				&=  \phi(x)+ \mathbb{E}^x\left[ \int_{0}^{\tau_{(0,R)}} e^{-\alpha t}(-\alpha-L)\phi(X_t)\d t \right].
			\end{split}
		\end{equation}
		Here, we used that $ -L $ is the generator of $X_t$. Note that $X_{\tau_{(0,R)}}\notin (0,R)$ and, thus, the first term in the previous equation equals zero by the compact support of $\phi$ in $(0,R)$. Let us restate \eqref{eq:dynkin} using the resolvent
		\begin{equation*}
			\phi(x)= G_\alpha(\alpha+L)\phi(x).
		\end{equation*}
		Let's integrate this against $\1_{(0,R)}$ which yields using the duality of the resolvent (recall that $L$ is symmetric in $L^2$), see \cite{Ber96}, for all nonnegative test functions $\phi\in C_c^\infty$
		\begin{align*}
			\int_{\R} \phi(x)\d x&=\int_{\R} \phi(x)\1_{(0,R)}(x)\d x=\int_{\R} G_\alpha(\alpha+L)\phi(x) \1_{(0,R)}(x)\d x\\
			&=\int_{\R} (\alpha+L)\phi(x) G_\alpha\1_{(0,R)}(x)\d x = \int_{\R} (\alpha+L)\phi(x) v_\alpha(x)\d x.
		\end{align*}
		Since $v_\alpha$ converges to $v$ and is uniformly bounded, the claim follows after taking the limit $\alpha\to 0$ by dominated convergence.

		This allows us to finalize the proof. Since $v$ satisfies the bounds claimed in the proposition and due to the interior regularity \cite[Theorem 2.4.3, Proposition 2.4.4]{FeRo24}, $v$ belongs to $C_{\loc}^{2s+\eps}((0,R))\cap C^s(\R)$ with an appropriate regularity estimate. A minor calculation reveals that $v$ is a weak solution. By uniqueness of weak solutions, see \cite[Lemma 2.3.3]{FeRo24}, the identification $v=\tilde{b}_R$ follows. 
	\end{proof}
	
	\section{The half line}\label{sec:half-line}
	The results in this section all revolve around problems involving operators $L\in \Gs$ in the half line $\R_+\coloneq (0,\infty)$. The construction of a nontrivial, one dimensional solution $b$ to the problem
	\begin{equation}\label{eq:one-dim-sol-halfspace}
		\begin{split}
			L b&=0 \text{ in }\R_+,\\
			b&=0\text{ on }\overline{\R}_-
		\end{split}
	\end{equation}	
	that behaves like $x_+^s$ is a vital step in the proof of the boundary regularity. It gives rise to a solution on the half space in higher dimensions and enables us to pinpoint the behavior of solutions to \eqref{eq:localised-equation} near flat boundaries. We benefit from the deep results obtained in one dimension using stochastic processes, see \cite[Proposition 3.5]{GrRy12} and \cite[Proposition 2.4]{BGR15}. Thereafter, a majority of this section deals with a one-dimensional Liouville theorem characterizing all half line solution under a certain growth assumption to be multiples of the specific solution to \eqref{eq:one-dim-sol-halfspace} constructed before. 	
	
	\begin{lemma}\label{lem:one-dim-sol-interval}
		Let $L \in \Gs$ and $b_1$ be the weak solution to 
		\begin{equation}\label{eq:one-dim-sol-interval}
			\begin{split}
				L b_1&=0 \text{ in }(0,1),\\
				b_1&=\1_{(1,\infty)}\text{ on }(0,1)^c.
			\end{split}
		\end{equation}	 
		The solution $b_1$ is nondecreasing and satisfies the two-sided bound
		\begin{equation*}
			C^{-1}x_+^s\le b_1(x)\le Cx_+^s
		\end{equation*}
		for all $x\in (0,1)$ and some constant $C$ that only depends on $s,\lambda$, and $\Lambda$.  
	\end{lemma}
	\begin{proof}
		The proof of the monotonicity relies on the sliding method, i.e., consider $x\mapsto b_1(x)-b_1(x-t)$ for some $0<t<1$, combined with the strong maximum principle, see \cite[Theorem 2.4.15 (a)]{FeRo24}. \smallskip
		
		For the moment, let $c_1$ be an arbitrary positive number to be chosen later. We define the function $v$ as $v\coloneq  c_1\tilde{b}_3+\1_{(2,\infty)}$ where $\tilde{b}_3$ is the function from \autoref{prop:torsion_estimate}. Then for any $x\in (0,1)$ in the weak sense
		\begin{align*}
			L v(x)= c_1-K((2-x,\infty))\begin{cases}
				\ge 0 &,\text{ if }c_1=K((1,\infty)),\\
				\le 0&, \text{ if }c_1=K((2,\infty)).
			\end{cases}
		\end{align*}
		Moreover, note that for $x\in (1,3)$ we know $v(x)\ge 1$ and
		\begin{equation*}
			v(x)\le c_1c_2\Lambda 3^s(3-x)^s+1\le c_1c_2 6\Lambda +1
		\end{equation*}
		where $c_2$ is the comparability constant from \autoref{prop:torsion_estimate}. The weak maximum principle applied to $b_1(x)- (c_1c_2 6\Lambda +1)^{-1}v$ with the choice $c_1=K((2,\infty))$ yields the lower bound in our claim. Finally, choosing $c_1=K((1,\infty))$ and applying the weak maximum principle to $v-b_1$ yields the upper bound. 
	\end{proof}
	
	As announced, the next goal is to construct a nontrivial half space solution. This is done by approximation and applying \autoref{lem:one-dim-sol-interval} to a rescaling of the approximating sequence. 
	
	\begin{proposition}\label{prop:one-dim-sol-halfline}
		Let $L \in \Gs$. There exists a positive constant $C=C(s,\lambda,\Lambda)$ and a distributional and strong solution $b\in \dot{C}^s(\R)\cap C_\loc^{3s}(\R_+)$ to \eqref{eq:one-dim-sol-halfspace} that satisfies the two-sided estimate
		\begin{equation}\label{eq:one-dim-sol-halfline-twosided-bound}
			C^{-1}x_+^s\le b(x)\le Cx_+^s, \quad x\in \R.
		\end{equation}
	\end{proposition}
	\begin{proof}
		We divide the proof into three steps. \smallskip
		
		\textit{Step 1.} In the first step, we construct an approximating sequence. We let for any $R>0$ the function $b_R$ be the weak solution to 
		\begin{equation*}
			\begin{split}
				L b_R &= 0 \text{ in }(0,R),\\
				b_R&=R^s\1_{(R,\infty)} \text{ on }(0,R)^c.
			\end{split}
		\end{equation*}
		
		\textit{Claim 1.} We find a uniform constant $c_1=c_1(s,\lambda,\Lambda)$ such that $c_1^{-1}x_+^s\le b_R(x)\le c_1 x_+^s$ for any $x\in (-\infty,R]$. 
		
		We rescale the function via $\bar{b}(x)\coloneq   R^{-s}b_R(Rx)$ which is a solution to \eqref{eq:one-dim-sol-interval} for some operator $\bar{L}\in \Gs$. By \autoref{lem:one-dim-sol-interval}, we find a constant $c_1$ such that $c_1^{-1}x_+^s\le \bar{b}(x)\le c_1 x_+^s$ for any $x\in (0,1)$. This yields the claim. \smallskip
		
		\textit{Claim 2.} We find a uniform constant $c_2=c_2(s,\lambda,\Lambda)$ such that $[b_R]_{C^s(\R)}\le c_2$. 
		
		We invoke the interior regularity result \cite[Theorem 2.4.3]{FeRo24} or \cite{FeRo24a} by which we find a constant $c_3$ such that
		\begin{equation*}
			\norm{u}_{C^{s}(B_{1/2})}\le c_3 \big( \|u\|_{L_{3s/2}^\infty(\R)} +\norm{f}_{L^\infty(B_1)}\big)
		\end{equation*}
		for any distributional solution $u$ to $L u=f$ in $B_1$. Scaling this interior regularity estimate and applying it to $b_R$ in an interval $B_r\subset (0,R)$ yields 
		\begin{equation*}
			r^s[b_R]_{C^{s}(B_{r/2})}\le c_3 \|b_R/(1+\abs{x}/r)^{3s/2}\|_{L^\infty(\R)} \le c_1 c_3r^s \|\abs{x}^s/(1+\abs{x})^{3s/2}\|_{L^\infty(\R)}\le c_1 c_3 r^s.
		\end{equation*}
		Thus, $b_R\in C_{\loc}^s((-\infty, R))$ with $[b_R]_{C^s((-\infty, R/2])}\le c_2$ for a uniform constant $c_2=c_2(s,\lambda,\Lambda)$. By repeating the process for $R^s-b_R(R-x)$, the claim follows. \smallskip
		
		\textit{Claim 3.} The set of functions $A\coloneq  \{u\in \dot{C}^s(\R)\mid u(0)=0\land [u]_{C^s(\R)}\le c_2 \}$ is a precompact subset of $L_{3s/2}^\infty(\R)$. 
		
		Let $u\in A$. By the growth bound, $\abs{u(x)}\le [u]_{C^s(\R)}\abs{x}^s$ the set $A$ is a subset of $L_{3s/2}^\infty(\R)$ and, clearly, the set $A$ is a bounded subset of $L_{3s/2}^\infty(\R)$. 
		
		The set of functions $A$ is clearly equicontinuous since they are uniformly bounded in $[\cdot]_{C^s(\R)}$. Now, let $x\in \R$. The set $\{ u(x)\mid u\in A \}\subset \R$ is relatively compact by Heine-Borel's theorem since it is a subset of $(-c_2\abs{x}^s,c_2\abs{x}^s)$ which is bounded. Due to a version of Arzel{\'a}-Ascoli, the set $A$ is precompact in $L_{3s/2}^\infty(\R)$ equipped with the topology of compact convergence, i.e., uniform convergence on each compact subset of $\R$. 
		
		Let $\{u_n\}\subset A$ be an arbitrary sequence of functions in $A$. By the previous paragraph, we find a subsequence $\{u_{m_n}\}$ and a limit $u_\infty\in \overline{A}$ which converges uniformly on each compact subset of $\R$. We claim that $\{u_{m_n}\}$	converges in the topology of $L_{3s/2}^\infty(\R)$. Let $\eps>0$ and fix $R\coloneq   (4c_2/\eps)^{2/s}$. Since $\{u_{m_n}\1_{[-R,R]}\}$ converges to $u_\infty\1_{[-R,R]}$ uniformly, we find $n\in \N$ such that $\norm{u_{m_n}-u_\infty}_{L^\infty([-R,R])}\le \eps/2$. Thus, 
		\begin{align*}
			\norm{u_{m_n}-u_\infty}_{L_{3s/2}^\infty(\R)}&\le \frac{\eps}{2}+\norm{u_{m_n}-u_\infty}_{L_{3s/2}^\infty([-R,R]^c)}\le \frac{\eps}{2}+2c_2\norm{\abs{\cdot}^s}_{L_{3s/2}^\infty([-R,R]^c)}\\
			&\le \frac{\eps}{2}+2c_2 R^{-s/2}=\eps.
		\end{align*}
		This proves the claim. \smallskip	
		
		\textit{Step 2.} Now, we pull the limit. Since $\{b_R\mid R>1\}$ is a precompact subset of $L_{3s/2}^\infty(\R)$ we find a sequence $b_{R_n}$ that converges to $b_{\infty}$ in $L_{3s/2}^\infty(\R)$, almost everywhere and $R_n\to \infty$ as $n\to \infty$. Clearly, $b_{\infty}$ satisfies the two-sided bound \eqref{eq:one-dim-sol-halfline-twosided-bound} since the approximating sequence does so locally uniformly. Moreover, $b_\infty\in C^s(\R)$. It remains to prove that $b_\infty$ is a distributional solution to the problem in the half space. 
		
		\textit{Claim 4.} For any $\phi\in C_c^\infty(\R_+)$, there exists a constant $c_5$ such that $\int_{B_r^c}\abs{L \phi(x)}\d x\le c_5 r^{-2s}$. \smallskip
		
		This claim follows directly from the first math display in the proof of \cite[Lemma 2.2.11]{FeRo24}.\smallskip
		
		We proceed with the main proof. By claim 2 and claim 3, we find a sequence $\{b_{R_n}\}$ and a limit $b_{\infty}$ in the set $A$ with respect to the $L_{3s/2}^\infty(\R)$-topology. Let $n$ be sufficiently large such that $\norm{b_{R_n}-b_\infty}_{L_{3s/2}^\infty(\R)}<\eps$. Set $r_0\coloneq   1+\distw{0}{\supp \phi}+\diam\supp\phi$. Now, due to claim 4 we find
		\begin{align*}
			\left|\int_{\R} b_\infty(x) L \phi (x)\d x\right|&=\left|\int_{\R_+} (b_{R_n}(x)-b_\infty(x)) L \phi (x)\d x\right| \le \eps\int_{\R_+} (1+\abs{x})^{3s/2} \abs{L \phi (x)}\d x\\
			&\le \eps\int_{[0,r_0]} (1+\abs{x})^{3s/2} \abs{L \phi (x)}\d x+\eps\sum_{k\in \N}4r_0^{3s/2}2^{3sk/2}\int_{2^{k-1}r_0}^{2^kr_0}\abs{L \phi (x)}\d x\\
			&\le  \eps \norm{\phi}_{C^2}\int_{[0,r_0]} (1+\abs{x})^{3s/2}\int_{\R}(1\wedge \abs{h}^2)K(\d h) \d x+8\eps r_0^{-s/2}\sum_{k\in \N}2^{-sk/2}.
		\end{align*}
		Since $\eps>0$ was arbitrary, $b_\infty$ is a distributional solution to \eqref{eq:one-dim-sol-halfspace}. It is also a strong solution by $b_\infty\in C^s(\R)$ and the interior regularity, see \cite[Proposition 2.4.4]{FeRo24}.
	\end{proof}
	
	The next step in our path to boundary regularity is a Liouville theorem in the half space under some growth assumptions. But, in order to prepare the Liouville theorem, we provide the following regularity estimate for the symbol $\mathcal{A}(\xi)$. 
	\begin{lemma}\label{lem:regularity_symbol}
		Let $L \in \Gs$. There exists a constant $C$ depending only on $s,\lambda,$ and $\Lambda$ such that the Fourier symbol $\mathcal{A}$ of $L$, satisfies
		\begin{equation*}
			\abs{\mathcal{A}(\xi)-\mathcal{A}(\xi+\eps)}\le C\begin{cases}
				\eps \abs{\eps+\xi}^{2s-1} & s> 1/2\\
				\eps \ln(e+\abs{\xi}/\eps) & s= 1/2\\
				\eps^{2s}& s<1/2
			\end{cases}
		\end{equation*}
		for any $\eps>0$ and any $\xi \in \R$.  
	\end{lemma}
	\begin{proof}
		Without loss of generality, we assume that $\xi >0$ since the proof in the case $\xi =0$ follows directly from the observation $\mathcal{A}(\eps)\le \Lambda \abs{\eps}^{2s}$. Now, note that 
		\begin{equation*}
			|\cos(\xi h)-\cos((\xi+\eps)h)|\le 2 \wedge \eps \abs{h}\big( 1\wedge  (\eps+\xi)\abs{h} \big).
		\end{equation*}
		This is a direct consequence of $\abs{\cos}\le 1$, the fundamental theorem of calculus, $\abs{\sin}\le 1$, and $\abs{\sin(t)}\le \abs{t}$. 
		
		\textit{Step 1.} The term $\I{}= 2\int_{1/\eps}^\infty K(\d h)$ is easily bounded using the property \eqref{X2a} by $\abs{\I{}}\le 2\Lambda \eps^{2s}$. 
		
		\textit{Step 2.} The second term is defined as 
		\begin{equation*}
			\II{}\coloneq   \int_{1/(\eps+\xi)}^{1/\eps} \eps h K(\d h).
		\end{equation*}
		Let $N,M\in \Z$ be such that $2^{-N-1}< \eps+\xi\le 2^{-N}$ and $2^{-M-1}< \eps\le 2^{-M}$. A dyadic decomposition reveals
		\begin{equation*}
			\II{}\le \eps \sum_{k=N}^M 2^{k+1} K\big((2^{k}, 2^{k+1})\big)\le 2\eps \Lambda\sum_{k=N}^M 2^{k(1-2s)}.
		\end{equation*}
		We distinguish three cases depending on the sign of the power $2s-1$. If $s<1/2$, then 
		\begin{equation*}
			\II{}\le 2\eps \Lambda \frac{2^{(M+1)(1-2s)} - 2^{N(1-2s)}}{2^{1-2s}-1}\le \frac{2^3\Lambda}{1-2s} \eps^{2s}.
		\end{equation*}
		If $s=1/2$, then 
		\begin{equation*}
			\II{}\le 2\eps \Lambda (M-N+1)\le 2\eps \Lambda \big(\ln(1/\eps)- \ln(1/(\eps+\xi))+2\big)\le 8\Lambda \eps \ln\left(e+\frac{\xi}{\eps}\right).
		\end{equation*}
		Finally, if $s>1/2$, then 
		\begin{equation*}
			\II{}\le 2\eps \Lambda \frac{2^{N(1-2s)}-2^{(M+1)(1-2s)}}{1-2^{1-2s}}\le \frac{8 \Lambda}{2s-1} \eps(\xi+\eps)^{2s-1}.
		\end{equation*}
		
		\textit{Step 3.} The last term $\III{}$ is given by 
		\begin{equation*}
			\III{}\coloneq   \int_0^{1/(\eps+\xi)} \eps(\eps+\xi) h^2 K(\d h).
		\end{equation*}
		Pick $N\in \Z$ such that $2^{N-1}\le \eps+\xi\le 2^{N}$. Using a dyadic decomposition, we write 
		\begin{align*}
			\III{}&\le \sum_{k=N}^{\infty} \eps(\eps+\xi) 2^{2(-k+1)}K((2^{-k},\infty))\le 4\eps(\eps+\xi)\Lambda\sum_{k=N}^{\infty} 2^{-k2(1-s)}\\
			&\le 4\eps(\eps+\xi) \Lambda \frac{2^{-(N-1)2(1-s)}}{1-4^{s-1}}\le \frac{4^3 \Lambda}{\ln(4)(1-s)}  \eps(\eps+\xi)^{2s-1}.
		\end{align*}
		If instead $N\le 0$, then we receive the same estimate by using an index shift. 
	\end{proof}
	
	The next lemma is a technical ingredient in the proof of the Liouville theorem on the half line, see \autoref{prop:liouville}, and allows us to rigorously apply the Sokhotski-Plemelj theorem. 
	\begin{lemma}\label{lem:pv-exists}
		Let $L\in \Gs$. The principal value integral
		\begin{equation*}
			\pv \int_{\R} \frac{\ln \mathcal{A}(\theta)}{\theta-\xi}\d \theta
		\end{equation*}
		exists, is continuous in the variable $\xi\in \R$, and it is bounded from above and below by a constant depending only on $s, \lambda$, and $\Lambda$.
	\end{lemma}
	\begin{proof}		
		Let's fix $\eps>0$ and assume without loss of generality that $\xi\ge 0$. We calculate
		\begin{align*}
			\int_{(-\eps,\eps)^c} \frac{\ln \mathcal{A}(\theta+\xi)}{\theta}\d \theta= \int_{\eps}^\infty \frac{\ln \mathcal{A}(\xi+\theta)-\ln \mathcal{A}(\xi-\theta)}{\theta}\d \theta=:\I{}+\II{}+\III{}.
		\end{align*}
		Here, the term $\I{}$ describes the part integrated on the interval $(\eps,\xi/2)$. Using the regularity estimate of $\mathcal{A}$ from \autoref{lem:regularity_symbol}, we treat the term 
		\begin{align*}
			\I{}\coloneq   \int_{\eps}^{\xi/2} \frac{\ln \mathcal{A}(\xi+\theta)-\ln \mathcal{A}(\xi-\theta)}{\theta}\d \theta
		\end{align*}
		as follows. Due to the regularity estimate for the symbol $\mathcal{A}$ in \autoref{lem:regularity_symbol}, we find a constant $C$ such that
		\begin{equation}\label{eq:log_estimate_symbol}
			\begin{split}
				|\ln \mathcal{A}(\xi+\theta)-\ln \mathcal{A}(\xi-\theta)| &\le \frac{\abs{\mathcal{A}(\xi+\theta) - \mathcal{A}(\xi-\theta) }}{\mathcal{A}(\xi-\theta)}\\
				&\le \frac{C \lambda^{-1}}{ (\xi-\theta)^{2s}}\begin{cases}
					\theta \abs{\xi+\theta}^{2s-1} &, s>1/2\\
					\theta\ln (e+\abs{\xi}/\theta)&, s=1/2\\
					\theta^{2s}&, s<1/2
				\end{cases}.
			\end{split}
		\end{equation}
		This allows us to estimate the term $\I{}$ in all three cases $s<1/2$, $s=1/2$, and $s>1/2$. Let's begin with $s>1/2$, then 
		\begin{align*}
			\abs{\I{}}\le C\int_{\eps}^{\xi/2}\frac{1}{\xi^{2s}} \abs{\xi+\theta}^{2s-1}\d \theta \le C \frac{(3/2)^{2s}}{2s}.
		\end{align*}
		If $s=1/2$, then 
		\begin{align*}
			\abs{\I{}}&\le C \int_{\eps}^{\xi/2} \frac{\ln(e+\abs{\xi}/\theta)}{\xi-\theta}\d \theta 
			= 4C\ln(e+1/2)-2C\frac{\ln(e+\xi/\eps)}{\xi/\eps}+2 C \int_{1/2}^{\xi/\eps} \frac{1}{x(e+x)}\d x\\
			&\le 4C\ln(e+1/2)+2C\int_{1/2}^{\infty} \frac{1}{x(e+x)}\d x.
		\end{align*}
		Lastly, if $s<1/2$, then 
		\begin{equation*}
			\abs{\I{}}\le C  \int_{\eps}^{\xi/2} \frac{\theta^{2s-1}}{(\xi-\theta)^{2s}}\d \theta \le \frac{C}{2s}. 
		\end{equation*}
		Now, we turn our attention to the second term 
		\begin{equation*}
			\II{}\coloneq   \int_{\xi/2}^{3\xi/2} \frac{\ln \mathcal{A}(\xi+\theta)-\ln \mathcal{A}(\xi-\theta)}{\theta}\d \theta.
		\end{equation*}
		A direct computation yields
		\begin{equation*}
			\abs{\II{}}\le 2 \int_{1/2}^{3/2} \left|\ln \frac{\mathcal{A}(\xi(1+\theta))}{\mathcal{A}(\xi(1-\theta))} \right|\d \theta\le 2 \int_{1/2}^{3/2} \left|\ln \frac{\lambda}{\Lambda}\frac{|1+\theta|^{2s}}{|1-\theta|^{2s}} \right|+ \left|\ln \frac{\Lambda}{\lambda}\frac{|1+\theta|^{2s}}{|1-\theta|^{2s}} \right|\d \theta.
		\end{equation*}
		In order to prove the claim, it remains to consider the term $\III{}$. Due to the symmetry of $A$, the estimate \eqref{eq:log_estimate_symbol} also holds with $\theta$ and $\xi$ swapped. Again, we need to distinguish three cases depending on the size of $s$. 
		
		If $s>1/2$, then 
		\begin{equation*}
			\abs{\III{}}\le C \int_{3\xi/2}^\infty \frac{\xi\abs{\xi+\theta}^{2s-1}}{\theta (\theta-\xi)^{2s}}\d \theta= C \int_{3/2}^\infty \frac{(1+\theta)^{2s-1}}{\theta (\theta-1)^{2s} }\d \theta<\infty. 
		\end{equation*} 
		If $s=1/2$, then, 
		\begin{align*}
			\abs{\III{}}&\le C \int_{3\xi/2}^\infty \frac{\xi\ln(e+\theta/\xi)}{\theta (\theta-\xi)}\d \theta= C \int_{3/2}^\infty \frac{\ln(e+x)}{x (x-1)}\d x<\infty.
		\end{align*}
		Lastly, if $s<1/2$, then 
		\begin{equation*}
			\abs{\III{}} \le C \int_{3\xi/2}^\infty \frac{\xi^{2s}}{\theta (\theta-\xi)^{2s}}\d \theta= C \int_{3/2}^{\infty} \frac{1}{\theta(\theta-1)^{2s}}\d \theta<\infty. 
		\end{equation*}
	\end{proof}
	
	We provide a version of the Sokhotski-Plemelj theorem for the convenience of the reader. 
	
	\begin{lemma}[{Sokhotski-Plemelj}]\label{lem:sokhotski}
		Let $A\in L^1(\R, 1/(1+|\theta|)^2 \d \theta)$ be an even function that is continuous at $\theta_0\in \R$ and such that the principal value integral $\pv \int_{\R} A(\theta)/(\theta-\theta_0)\d \theta$ exists. Then the complex function $a:\C\setminus \R\to \C$
		\begin{equation*}
			a(\xi)\coloneq \lim_{R\to \infty}\int_{-R}^R \frac{A(\theta)}{\theta-\xi}\d \theta, \quad (\xi \in \C\setminus \R)
		\end{equation*}
		 has the limits
		\begin{equation*}
			\begin{aligned}
				\lim\limits_{\eps\to 0}a(\theta_0+i\eps) &= \pi i A(\theta_0) + \pv \int_{\R} \frac{A(\theta)}{\theta-\theta_0}\d \theta,\\
				\lim\limits_{\eps\to 0}a(\theta_0-i\eps) &= -\pi i A(\theta_0) + \pv \int_{\R} \frac{A(\theta)}{\theta-\theta_0}\d \theta.
			\end{aligned}
		\end{equation*}
	\end{lemma}
	In the literature this lemma is typically stated for Schwartz functions or Hölder continuous functions, see e.g. \cite[Section 5, Lemma 5.1]{Esk81}, but it generalizes easily to our case. We provide a short and simple proof in order to be self contained. 	
	\begin{proof}
		We decompose
		\begin{equation*}
			\frac{1}{\theta-(\theta_0+i\eps)}= \frac{\theta-\theta_0}{(\theta-\theta_0)^2+\eps^2}+ \frac{i\eps}{(\theta-\theta_0)^2+\eps^2}.
		\end{equation*}
		Since $A$ is even and $A\in L^1(\R,1/(1+|\theta|)^2\d \theta)$, the term $a(\theta)$ exists and is continuous in $\C\setminus \R$. Clearly, the real part of $a(\theta_0+i\eps)$ is responsible for the $\pv$-integral in the limit $\eps \to 0$. Since $\frac{1}{\pi} \frac{\eps}{\theta^2+\eps^2}$ is the Poisson kernel in the half space $(\theta,\eps)\in \R^2_+$, the convergence of the imaginary part of $a(\theta_0+i\eps)$ follows from the continuity of $A$ at $\theta_0$. By complex conjugation, the result for $a(\theta-i\eps)$ follows. 		
	\end{proof}
	
	In order to prove the Liouville theorem, we use Fourier analysis. The following lemma allows us to characterize the Fourier transform of a function under specific regularity assumptions and growth bounds in sense of tempered distributions. It is decisive in order to make the proof of \autoref{prop:liouville} rigorous.
	\begin{lemma}\label{lem:distribution-estimates}
		Let $u:\R\to \R$ be a function such that for some $s\in (0,1)$ and $\eps\in (0,s/2)$ the following three properties hold
		\begin{align}
			|u(x)|&\le C |x|^s\text{ for }|x|\le 1,\\ 
			|u(x)|&\le C(1+|x|)^{2s-\eps} \text{ for all }x,\\
			[u]_{C^{3s}(R,2R)}&\le C R^{-s-\eps}\text{ for all }R\ge 1.
		\end{align}
		Then we find a function $w:\R\to \C$ satisfying 
		\begin{align*}
			|w(\xi)|&\le C |\xi|^{-1-2s+\eps} \text{ for }|\xi|\le 1,\\
			|w(\xi)|&\le C |\xi|^{\eps-2s} \text{ for }|\xi|\ge 1
		\end{align*}
		and two complex numbers $a_0$ and $a_1$ such that the Fourier transform of $u$ in terms of tempered distributions $\sF u$ can be written as
		\begin{equation*}
			\sF u = w + \sum_{k=0}^{\lfloor 2s-\eps \rfloor} a_k \partial^k \delta_0.
		\end{equation*} 
		Here, $\delta_0$ is the Dirac-delta distribution in the origin. 
	\end{lemma}
	\begin{proof}
		Let $\xi\in \R\setminus\{0\}$. We fix a dyadic decomposition. Let $\psi\in C_c^\infty(\R)$ be a function supported on $[-2,-1/2]\cup[1/2,2]$ such that $\sum_{n\in \Z} \psi(2^{-n} \xi)=\1_{\R\setminus \{0\}}(\xi)$. Instead of studying the Fourier transform of $u$ directly, we consider $u_n(x)\coloneq  u(x) \psi(2^{-n} x)$ and study the Fourier transform of the integrable function $v_N\coloneq  \sum_{n\le N} u_n$. Note that $v_N\to u$ locally uniformly on $\R$. We distinguish two cases. 
		
		\textit{Case 1.} If $n$ satisfies $2^n\le 1/|\xi|$, then we estimate
		\begin{equation*}
			|\hat{u}_n(\xi)|=\left|\left(  \int_{2^{n-1}}^{2^{n+1}}+ \int_{-2^{n+1}}^{-2^{n-1}}\right) e^{-i\xi x} \psi(2^{-n} x) u(x)\d x\right|\le C \begin{cases}
				2^{(s+1)n} & \text{ if } n\le 0,\\
				2^{(2s-\eps+1)n} &\text{ if } n>0.
			\end{cases} 
		\end{equation*}
		Here, we used that $|u(x)|\le C x_+^s$ for $x<1$ and $|u(x)|\le Cx^{2s-\eps}$ for $x\ge 1$. Summing over $n$ under the assumption $2^n\le 1/|\xi|$ yields
		\begin{align*}
			\left|\sum_{2^n\le |\xi|^{-1}}\hat{u}_n(\xi)\right|&\le C\sum_{2^n\le |\xi|^{-1}\wedge 1 } 2^{(s+1)n} + C\sum_{|\xi|^{-1}\wedge 1<2^n\le |\xi|^{-1}} 2^{(2s-\eps+1)n}\\
			&\le C\begin{cases}
				|\xi|^{-s-1}	& \text{ if }|\xi|\ge 1,\\
				|\xi|^{-2s+\eps-1}  &\text{ if }|\xi|<1.
			\end{cases}
		\end{align*}
		
		\textit{Case 2.} If $n$ satisfies $2^n>1/|\xi|$, then we need to use the interior regularity of $u$ to obtain decay. We define the function $v_n(x)\coloneq   u_n(2^{n}x)$. Note that $v_n$ is supported on $[1/2,2]$ and 
		\begin{align*}
			\norm{v_n}_{L^\infty}&\le C 2^{ns},\\
			[v_n]_{C^{3s}}&\le C 2^{(2s-\eps)n}.
		\end{align*}
		Here, we used \eqref{eq:int_reg_for_u}. Moreover, we clearly have $\hat{u}_n(\xi)= 2^n \hat{v}_n(2^n \xi)$. 
		
		Notice that $v_n$ is a $C^{3s}$-regular function supported in the compact set $[1/2,2]$. Using the identity
		\begin{equation*}
			2\hat{v}_n(\xi)= \int\left( v_n(x)-v_n\left(x+\frac{\pi}{\xi}\right) \right)e^{-ix\xi}\d x,
		\end{equation*}
		we can bound the Fourier transform via finite differences 
		\begin{equation*}
			|\hat{v}_n(\xi)|\le C \abs{\xi}^{-3s}\norm{v_n}_{C^{3s}}\le C \abs{\xi}^{-3s}2^{(2s-\eps)n}.
		\end{equation*}
		Using the definition of $v_n$, this estimate yields 
		\begin{equation*}
			|\hat{u}_n(\xi)|=2^n|\hat{v}_n(2^n \xi)|\le C \abs{\xi}^{-3s}2^{-(s+\eps)n}.
		\end{equation*}
		We complete this case by summing over all $n$ up to $N\in \N$ such that $2^n\ge 1/|\xi|$. This yields
		\begin{equation*}
			\bigg|\sum\limits_{\substack{2^n\ge 1/|\xi|\\ n\le N}}\hat{u}_n(\xi)\bigg|\le C |\xi|^{-3s} \sum_{2^n\ge 1/|\xi|} 2^{-(s+\eps)n}\le C |\xi|^{\eps-2s}.
		\end{equation*}
		
		Finally, summing over all $n$ up to $N\in \N$ yields
		\begin{equation}\label{eq:v-N-estimate}
			|\sF v_N(\xi)|\le C\begin{cases}
				|\xi|^{-1-2s+\eps}&, |\xi|\le 1,\\
				|\xi|^{\eps-2s}&, |\xi|>1.
			\end{cases}
		\end{equation}
		Note for any $\phi\in \mathcal{S}(\R)$
		\begin{equation*}
			\int_{\R} \hat{v}_N(\xi) \phi(\xi)\d \xi = \int_{\R} v_N(\xi)\hat{\phi}(\xi)\d \xi \to \int_{\R} u(\xi)\hat{\phi}(\xi)\d \xi, \text{ as }N\to\infty, \text{ i.e.,}
		\end{equation*}
		$\sF v_N\to \sF u$ as $N\to \infty$ in the sense of tempered distributions. \smallskip

		The previous estimates yield the locally uniform convergence on $\R\setminus \{0\}$ of the sum 
		\begin{equation*}
			w(\xi)\coloneq \sum_{n\in \Z} \hat{u}_n(\xi).
		\end{equation*}
		We define $w(0)\coloneq 0$. Clearly, this function $w$ satisfies the upper bound \eqref{eq:v-N-estimate}. \smallskip  
		
		Since $w\notin L_\loc^1(\R)$ this function does not directly yield a tempered distribution. Instead we define $W$ for $\phi\in \mathcal{S}(\R)$ in the case $2s-\eps<1$ via
		\begin{equation}\label{eq:def-W}
			W[\phi]\coloneq \int w(\xi)\big( \phi(\xi)-\phi(0)\chi(\xi) \big)\d \xi
		\end{equation}
		where $\chi\in C_c^\infty(\R)$ is a smooth, even cutoff with $\phi=1$ on $[-1,1]$. In the case $2s-\eps\ge 1$, we set 
		\begin{equation*}
			W[\phi]\coloneq \int w(\xi)\big( \phi(\xi)-(\phi(0)+\xi\phi'(0))\chi(\xi) \big)\d \xi.
		\end{equation*}
		Note that if $\phi$ is not supported in $\{0\}$, then $W[\phi]=\int w\phi $. 
		
		\textit{Claim.} The support of the tempered distribution $\sF u - W$ is contained in $\{0\}$. 
		
		Let $A\subset \R\setminus\{0\}$ be open, bounded and $\phi \in \mathcal{S}(\R)$ such that $\supp\phi\subset A $. Since $\hat{v}_N\to \sF u$ in $\mathcal{S}'(\R)$ we estimate 
		\begin{align*}
			| (\sF u - W)[\phi]| &= |\lim\limits_{N\to \infty}\int (\hat{v}_N(\xi)-w(\xi))\phi(\xi)\d \xi= \lim\limits_{N\to \infty}|\sum_{n=N+1}^\infty\int u_n(\xi)\hat{\phi}(\xi)\d \xi|\\
			&\le \lim\limits_{N\to \infty}\int_{(-2^{N-1}, 2^{N-1})^c} |u(\xi)\hat{\phi}(\xi)|\d \xi=0
		\end{align*}
		because $\hat{\phi}$ is a Schwartz function and, thus, $u \hat{\phi}$ is integrable. This proves the claim.
		
		It remains to identify the distribution $\sF u$ in the origin.
		
		Due to \cite[Proposition 2.4.1]{Gra14}, we know that the tempered distribution $\sF u -W$, since it is supported on $\{0\}$, must be a finite sum of Dirac-delta distributions and its derivatives, i.e., we find $K\in \N$ and complex numbers $a_0,\dots, a_k$ such that 
		\begin{equation*}
			\sF u - W = \sum_{k=0}^{K} a_k \delta_0^{(k)}.
		\end{equation*}		
		Let's distinguish two cases. \smallskip
		
		If $2s-\eps<1$, then for a Schwartz function $\phi\in \mathcal{S}(\R)$ supported in $[-1,1]$ that vanishes at the origin of order $1$ we find
		\begin{align*}
			(\sF u - W)[\phi]= \lim\limits_{N\to \infty} \int_{-1}^1 \big(\hat{v}_N(\xi)-w(\xi)\big) \phi(\xi)\d \xi. 
		\end{align*}
		Due to \eqref{eq:v-N-estimate} and since $\phi$ vanishes of order $1$, we may estimate
		\begin{equation*}
			|\big(\hat{v}_N(\xi)-w(\xi)\big) \phi(\xi)\1_{(-1,1)}(\xi)|\le 2C |x|^{-2s+\eps}\1_{(-1,1)}(\xi)=:g(\xi).
		\end{equation*}
		Note that the function $g$ is integrable and dominates the sequence of functions. By dominated convergence, we know that $\int_{-1}^1 \big(\hat{v}_N(\xi)-w(\xi)\big) \phi(\xi)\d x$ converges to zero as $N\to \infty$ since $w$ is the pointwise limit of $\hat{v}_N$ almost everywhere. Thus, $(\sF u-W)[\phi]=0$ for any distribution vanishing of order $1$. This immediately yields that $a_1=\dots = a_K=0$. \smallskip
		
		If $1\le 2s-\eps<2$, then we repeat the above procedure with a Schwartz function that vanishes of order $2$ to deduce that $a_2=\dots=a_K=0$. 
	\end{proof}
	
	In order to apply \autoref{lem:distribution-estimates} to a solution of \eqref{eq:one-dim-sol-halfspace}, we check the prerequisites in the next lemma. 
	\begin{lemma}\label{lem:int_reg_for_u}
		Let $L\in \Gs$. If $u\in L_{2s-\eps}^\infty(\R)$ is a continuous distributional solution of \eqref{eq:one-dim-sol-halfspace}, then $u\in C^{3s}_{\loc}(\R_+)$ and we find a constant $C$ such that 
		\begin{equation}\label{eq:cs-one-dim-estimate}
			|u(x)|\le C x_+^s
		\end{equation}
		for all $x\in [0,1]$ and
		\begin{equation}\label{eq:int_reg_for_u}
			[u]_{C^{3s}(R,2R)}\le C\frac{(1+R)^{s-\eps}}{R^{2s}}
		\end{equation}
		for all $R>0$ if $s\ne 1/3$. If $s=1/3$, then $[u]_{C^{1-\eps/2}(R,2R)}\le C R^{-1/3-\eps/2}$.
		
	\end{lemma}
	\begin{proof}
		Consider the modified function $\tilde{u}\coloneq   u\1_{(0,2)}$. This function satisfies for any $x\in (0,1)$
		\begin{align*}
			\abs{L \tilde{u}(x)}&= |\int_{\R} u(x+h)\1_{[2,\infty)}(x+h)K(\d h)| \le C \int_{\R} (1+\abs{x+h})^{2s-\eps}\1_{[2,\infty)}(x+h)K(\d h)\\
			&\le  C \sum_{k\in \N} 2^{1+(2s-\eps)k}\int_{2^{k-1}}^{2^k} K(\d h)\le \Lambda C \frac{2^{1+2s} }{1-2^{-\eps}}.
		\end{align*}
		Moreover, for any $x\in [1,\infty)$
		\begin{align*}
			\tilde{u}(x)\le C(1+\abs{x})^{2s-\eps}\1_{(1,2)}(x)\le C3^{2s-\eps}\1_{(1,2)}(x).
		\end{align*}
		Let $\tilde{b}_{3}$ be the torsion function in $(0,3)$. By \autoref{prop:torsion_estimate}, we find a constant $c_1$ such that $\tilde{b}_{3}(x)\ge c_1^{-1}4^s(x\wedge 3-x)^s\ge c_1^{-1}4^s$ for any $x\in (1,2)$. An application of the weak maximum principle to the function $(\Lambda C \frac{2^{1+2s} }{1-2^{-\eps}} + c_14^{-s}C3^{2s-\eps})\tilde{b}_{3} -\tilde{u}$ yields the estimate \eqref{eq:cs-one-dim-estimate}.  \smallskip
		
		Together with the growth bound, we find 
		\begin{equation}\label{eq:aux-growth-estimate}
			|u(x)|\le C x_+^s (1+|x|)^{s-\eps}.
		\end{equation}

		Next, we need quantified $C^{s}$-regularity estimate. 
		
		\textit{Claim 1.} The function $u$ belongs to $C_{\loc}^{s}(\R)$. 
		
		This claim is a direct consequence of the estimate \eqref{eq:aux-growth-estimate} and the interior regularity \cite[Theorem 2.4.3]{FeRo24}.\smallskip
		
		\textit{Claim 2.} We find a constant $C$ such that for all $R>0$
		\begin{equation*}
			[u]_{C^s([R,2R])}\le C (1+R)^{s-\eps}.
		\end{equation*}
		
		In order to prove this claim, we define a new function. Let $\eta$ be a smooth compactly supported bump function such that $\eta=1$ on $[1/4,4]$ and $\eta=0$ on $[1/8,8]^c$. We define $v(x)\coloneq  u(Rx)\eta(x)$. This function satisfies for $x\in (1/4,4)$
		\begin{equation*}
			\tilde{L} v(x)= 0+\int_\R u(R(x+h)) (1-\eta(x+h))\tilde{K}(\d h).
		\end{equation*}	
		Here, $\tilde{L}\in \mathcal{G}_s(\lambda,\Lambda)$ is driven by the L{\'e}vy-measure $\tilde{K}= R^{2s}K(R (\cdot))$, see \cite[Remark 2.1.19]{FeRo24}. Since $\eta=1$ on $[1/4,4]$, we estimate for $x\in (1/2,5/2)$
		\begin{align*}
			|\tilde{L} v(x)|\le \Big( \int_{4-x}^\infty +\int_{-x }^{1/4-x} \Big)|u(R(x+h))|\tilde{K}(\d h)=\I{}+\II{}.
		\end{align*}
		To estimate $\I{}$, we use a dyadic decomposition of the integral. Note that due to \eqref{eq:aux-growth-estimate}
		\begin{align*}
			\I{}\le \sum_{n=1}^\infty \int_{2^{n}}^{2^{n+1}} C(Rh)^s(1+Rh)^{s-\eps} \tilde{K}(\d h)\le 2^4C\Lambda R^s(1+R)^{s-\eps}  \sum_{n=1}^\infty 2^{-n\eps}.
		\end{align*}
		Due to \eqref{eq:aux-growth-estimate}, the term $\II{}$ is trivially bounded from above by $CR^s(1+R)^{s-\eps}$. 
		
		The interior regularity estimate \cite[Theorem 2.4.3]{FeRo24} applied to the solution $v$ yields
		\begin{equation*}
			R^s[u]_{C^s([R,2R])}=[v]_{C^s([1,2])}\le \norm{v}_{C^{2s-\eps'}([1,2])}\le C\big( \norm{v}_{L_{2s-\eps}^\infty(\R)} +C R^s(1+R)^{s-\eps} \big).
		\end{equation*}
		Note that $\norm{v}_{L_{2s-\eps}^\infty(\R)}\le \norm{u}_{L^\infty(R/8,8R)}\le C R^s(1+R)^{s-\eps}$ where we used \eqref{eq:aux-growth-estimate} again. This yields the claim.\medskip 
		
		Now, we are in the position to prove \eqref{eq:int_reg_for_u}. We prove the estimate for $s\ne 1/3$ since the proof in the case $s=1/3$ follows analogously. Let $R>0$ and, again, we consider the function $v$ as in claim 2. Instead of the size of the inhomogeneity $\tilde{L}v$, we need to bound $\|\tilde{L}v\|_{C^s([1/2,5/2])}$ in order to apply \cite[Proposition 2.4.4]{FeRo24}. Let $x,y\in [1/2,5/2]$, then 
		\begin{align*}
			|\tilde{L}v(x)-\tilde{L}v(y)|&\le \Big(\int_{3/2}^\infty+\int_{-2/5}^{-1/4}\Big) \Big( |u(R(x+h))(\eta(y+h)-\eta(x+h))| \\
			&\quad+ |u(R(x+h))-u(R(y+h))|(1-\eta(y+h)) \Big) \tilde{K}(\d h)=:\III{}+\IV{}.
		\end{align*}
		Since $\eta$ is smooth, we simply estimate the first term via 
		\begin{equation*}
			\III{}\le |x-y|2C\int_{1/4}^\infty  R^s(1+R)^{s-\eps}|h|^{2s-\eps}\tilde{K}(\d h)\le C|x-y| R^s(1+R)^{s-\eps}.
		\end{equation*}
		Moreover, we use claim 2, $u\in C_{\loc}^s(\R)$ and estimate the second term as follows.
		\begin{equation*}
			\IV{}\le C R^s(1+R)^{s-\eps}|x-y|^s  \int_{1/4}^\infty |h|^{s-\eps}\tilde{K}(\d h).
		\end{equation*}
		This allows us to use the higher order interior regularity estimate \cite[Proposition 2.4.4]{FeRo24} for $\tilde{L}v$ which yields
		\begin{equation}\label{eq:int-C3s-aux1}
			\begin{split}
				R^{3s}[u]_{C^{3s}([R,2R])}&= [v]_{C^{3s}([1,2])}\le C\big(\norm{v}_{C^s(\R)}+ \|\tilde{L}v\|_{C^s([1/2,5/2])} \big)\\
				&\le C\big( R^s(1+R)^{s-\eps}+ [v]_{C^s(\R)}+ [\tilde{L}v]_{C^s([1/2,5/2])}  \big).
			\end{split}
		\end{equation}
		Here, we use a similar estimate as in the proof of claim 2 in order to bound the absolute value of $v$ and the inhomogeneity. It remains to consider the seminorms. Note that, since $v$ is supported in $[1/8,8]$,
		\begin{align*}
			[v]_{C^s(\R)}&= \sup_{\substack{x\in [1/8,8]\\y\in \R}}\frac{|u(Rx)\eta(x)-u(Ry)\eta(y)|}{|x-y|^{s}}\le \sup_{\substack{x\in [1/8,8]\\y\in \R}}|u(Rx)|\frac{|\eta(x)-\eta(y)|}{|x-y|^{s}} \\
			&\quad +\sup_{\substack{x\in [1/8,8]\\y\in \R}}|\eta(y)|\frac{|u(Rx)-u(Ry)|}{|x-y|^{s}} \le CR^s(1+R)^{s-\eps}[\eta]_{C^s(\R)}+ R^s[u]_{C^s([R/8,8R])}\\
			&\le CR^s(1+R)^{s-\eps}.
		\end{align*}
		Here, we used \eqref{eq:aux-growth-estimate} and claim 2. It remains to consider $[\tilde{L}v]_{C^s([1/2,5/2])}$. Due to the above estimates on $\III{}$ and $\IV{}$, we find $[\tilde{L}v]_{C^s([1/2,5/2])}\le CR^s(1+R)^{s-\eps}$. \smallskip		
		
		Combining the estimate for both seminorms with \eqref{eq:int-C3s-aux1} yields
		\begin{equation*}
			[u]_{C^{3s}([R,2R])}\le C\frac{(1+R)^{s-\eps}}{R^{2s}},
		\end{equation*}
as wanted.
	\end{proof}
	We have collected all necessary auxiliary results that allow us to prove a higher order Liouville result in the half line. 	
	\begin{proposition}\label{prop:liouville}
		Let $L\in \Gs$ and $\eps>0$. If $u\in L_{2s-\eps}^\infty(\R)$ is a continuous distributional solution of \eqref{eq:one-dim-sol-halfspace}, then we find a constant $C$ such that $u=Cb$ where $b$ is the function from \autoref{prop:one-dim-sol-halfline}. 
	\end{proposition}
	\begin{proof}
		Before we start with the main part of the proof, we prove a few auxiliary properties of the solution $u$. Without loss of generality, we assume that $\eps<s/2$. 
		
		\textit{Claim 1.} The function $u$ is also a solution in the pointwise sense and in sense of a tempered distribution, i.e., when tested against Schwartz functions. 
		
		The first part is a direct consequence of the interior regularity \autoref{lem:int_reg_for_u} together with the growth estimate and \eqref{eq:cs-one-dim-estimate}. The second part is simply proven via approximation since the function $u$ is not growing faster than $\abs{x}^{2s-\eps}$. This proves the claim. \smallskip
		
		\textit{Step 1.} Now, we setup the Wiener-Hopf equation.
		
		We would like to proceed by taking the Fourier-transform of both $u$ and $L u$. This is not possible in the classical sense. Note that $u\in L_{\loc}^1(\R)$ with at most polynomial growth. Moreover, since $Lu=0$ in $(0,\infty)$ and $u(x)\le C x_+^s (1+|x|)^{s-\eps}$, we also know $L u \in L_{\loc}^1(\R)$ with at most polynomial growth outside of the origin. Thus, we may take the Fourier transform in the sense of tempered distributions. Since $u =0$ in $(-\infty,0]$, the complex function $\C_-\ni\xi\mapsto \hat{u}(\xi)$ is holomorphic in $\C_-$. In a similar fashion, $\widehat{L u}$ can be extended to a holomorphic function in $\C_+$. It may also be viewed as a tempered distribution which is continuous in the sense that for $\xi \in \C_-$
		\begin{equation*}
			\hat{u}(\xi)\to \sF u(\Re\xi) \text{ as }\Im\xi \to 0-, \text{ i.e.,}
		\end{equation*}
		\begin{equation*}
			\int_{\R} \hat{u}(x-i\eps)\phi(x)\d x \to \int_{\R}u(x)\hat{\phi}(x)\d x.
		\end{equation*}
		
		\textit{Claim 2.} $\sF L u[\phi]=u[\mathcal{A} \hat{\phi}]=\mathcal{A}\sF  u[\phi]$ for any Schwartz function $\phi\in \mathcal{S}(\R)$. 
		
		For any Schwartz function $\phi\in \mathcal{S}(\R)$, we have due to the symmetry of $L$ in $L^2(\R)$
		\begin{align*}
			\sF L u[\phi]&= \int_{\R} L u(x)\overline{\check{\phi}}(x)\d x=\int_\R  u(x) \overline{L\check{\phi}(x)}\d x\\
			&= \int_\R  u(x) \overline{\sF_{\xi}^{-1}\mathcal{A}(\xi)\phi(\xi)}(x)\d x=  \mathcal{A}(\xi)\sF u[\phi].
		\end{align*}
		This proves the claim.\smallskip
		
		Let us denote $f\coloneq  L u$ for the distributional solution on $\R$ to improve readability. Then claim 2 boils down to the Wiener-Hopf equation 
\[\mathcal{A}(\xi)\sF u= \sF f \quad \textrm{in}\quad \mathcal{S}^\prime.\]
		
		\textit{Step 2.} Now, we use the Wiener-Hopf factorization to decompose the symbol $\mathcal{A}(\xi)$. 
		
		Let us first make the observation that the symbol $\mathcal{A}$ is Hölder continuous to some order by \autoref{lem:regularity_symbol}. This is crucial in the following steps. 
		
		We define two functions $a_+:\C_+\to \C$ and $a_-:\C_-\to \C$ via
		\begin{align*}
			a_+(\xi)&\coloneq   \frac{1}{2\pi i}\int_{\R} \frac{\ln \mathcal{A}(\theta)}{\theta-\xi}\d \theta, \qquad \Im \xi >0,\\
			a_-(\xi)&\coloneq   \frac{-1}{2\pi i}\int_{\R} \frac{\ln \mathcal{A}(\theta)}{\theta-\xi}\d \theta,\qquad  \Im \xi <0.
		\end{align*}
		By Sokhotski-Plemelj's theorem, see \autoref{lem:sokhotski}, \autoref{lem:pv-exists}, and \autoref{lem:regularity_symbol}, we find that for any $\xi \in \R\setminus \{0\}$ in the limit $\eps \to 0^+$
		\begin{align*}
			a_+(\xi+i\eps)&\longrightarrow \frac{1}{2}\ln \mathcal{A}(\xi)+\frac{1}{2\pi i}\pv \int_{\R} \frac{\ln \mathcal{A}(\theta)}{\theta-\xi}\d \theta, \\
			a_-(\xi-i\eps)&\longrightarrow \frac{1}{2}\ln \mathcal{A}(\xi)-\frac{1}{2\pi i}\pv \int_{\R} \frac{\ln \mathcal{A}(\theta)}{\theta-\xi}\d \theta.
		\end{align*}
		
		Thus, the functions $\mathcal{A}_+\coloneq  e^{a_+}$, $\mathcal{A}_-\coloneq   e^{a_-}$ factorize the symbol $\mathcal{A}(\xi)=\mathcal{A}_+(\xi) \mathcal{A}_-(\xi)$ on the real line $\xi\in \R$. \smallskip
		
		\textit{Step 3.} We define the tempered distribution $J\coloneq   \mathcal{A}_-(\xi)\sF u$ on the real line. This auxiliary distribution is the key to our result. Note that we can extend it continuously to $\C_-$. 
		
		\textit{Claim 3.} The map $J$ defines a tempered distribution. 
		
		In order to prove this claim, it is sufficient to prove the boundedness $J[\phi]\le \norm{\phi}_{\sS(\R)}$.  Let $\phi\in \sS(\R)$. Note that $\mathcal{A}_-=\sqrt{\mathcal{A}}$ satisfies the two-sided bound $|\mathcal{A}_-(\xi)|\approx \abs{\xi}^s$ on the real line $\xi \in \R$. We define two auxiliary functions. $B(\xi)= \mathcal{A}_-(\xi)/(1+\abs{\xi})^4$ and $\eta(\xi)\coloneq   (1+\abs{\xi})^4 \phi(\xi)$. Clearly, $B\in L^1(\R)$ and, thus, $\hat{B}\in L^\infty(\R)$. Then
		\begin{align*}
			J[\phi]&= \int_{\R} u(x) \widehat{ \mathcal{A}_- \phi }(x)\d x=  \int_{\R} u(x) \hat{ B} \ast \hat{ \eta }(x)\d x\\
			&\le C\int_{\R} (1+\abs{x})^{2s-\eps}\int_{\R} |\hat{B}(y)|\, |\hat{\eta}(x-y)| \d y \d x\\
			&\le C \norm{\hat{\eta}}_{L^\infty(\R)} \int_{\R} (1+\abs{x})^{2s-\eps}\int_{\R} |\hat{B}(y)|\, (1+\abs{x-y})^{-4} \d y \d x.
		\end{align*}
		Now, a minor calculation yields
		\begin{equation*}
			\int_{\R} \frac{(1+\abs{x})^{2s-\eps}}{(1+\abs{x-y})^{4}}\d x\le C\frac{1}{(1+\abs{y})^{3-2s+\eps}}.
		\end{equation*}
		Thus, we deduce by the continuity of the Fourier transform on the space of Schwartz functions
		\begin{align*}
			J[\phi]&\le C \norm{\eta}_{L^1(\R)} \int_{\R} (1+\abs{y})^{2s-\eps-3}|\hat{B}(y)|\d y \\
			&\le C  \norm{\phi}_{\sS(\R)} \norm{B}_{L^1(\R)}\int_{\R} (1+\abs{y})^{2s-\eps-3}\d y.
		\end{align*}
		This proves the claim. 
		
		\textit{Claim 4.} The map $I\coloneq   \mathcal{A}_+(\xi)^{-1} \sF f$ defines a tempered distribution. 
		
		The proof is analogous to that of claim 3. 
						
		\textit{Claim 5.} In sense of tempered distributions, the equality $J=I$ holds.
		
		Since both $J$ and $I$ define tempered distributions by claim 3 and claim 4, the claim follows directly from claim 2. 
		
		By claim 5, the identity $J= \mathcal{A}_-(\xi)\sF u = 1/\mathcal{A}_+(\xi)  \sF f$ holds. Thus, the distribution $J$ can be extended analytically to $\C_+$ and $\C_-$. 
		
		\textit{Claim 6.} There exists a function $w:\R\setminus\{0\}\to \R$ and two complex numbers $a_0$ and $a_1$ such that the tempered distribution $\sF u$ is given by $w$ outside of the origin, i.e., 
		\begin{equation*}
			\sF u[\phi]= \int w(\xi)\phi(\xi)\d \xi + \sum_{k=0}^{\lfloor 2s-\eps \rfloor}a_k\phi^{(k)}(0).
		\end{equation*} 
		Moreover, the function $w$ satisfies 
		\begin{align*}
			|w(\xi)| \le C \abs{\xi}^{-1-2s+\eps} \text{ for small }\xi,\\
			|w(\xi)| \le C (1+\abs{\xi})^{\eps-2s} \text{ for large }\xi. 
		\end{align*}
		
		This claim directly follows from \autoref{lem:distribution-estimates}. \smallskip 
		
		\textit{Claim 7.} The tempered distribution $\xi J$ satisfies 
		\begin{equation*}
			\xi J [\phi]= \xi \mathcal{A}_-(\xi)\sF u[\phi]= \int w(\xi)\mathcal{A}_-(\xi)\xi \phi(\xi)\d \xi.
		\end{equation*}
		
		This claim is a direct consequence of claim 6. Note that
		\begin{equation*}
			\mathcal{A}_-(\xi)\xi \delta_0=0
		\end{equation*}
		and 
		\begin{equation*}
			\mathcal{A}_-(\xi)\xi \delta_0'= -\mathcal{A}_-(\xi)\delta_0=0.
		\end{equation*}
		Here, we used that distributions satisfy the product rule and, thus, $\xi \delta_0' = -\delta_0$.\smallskip

		\textit{Claim 8.} The function $\xi \mapsto \xi \mathcal{A}_-(\xi) w(\xi)$ has a holomorphic extension to $\C$. 
		
		Note that the distribution $J$ has an analytic extension to $\C\setminus \R$ which is continuous on $\C\setminus \{0\}$ and $J=\xi \mathcal{A}_-(\xi)w(\xi)$ on $\R\setminus \{0\}$. Thus, the function $\xi \mathcal{A}_-(\xi) w(\xi)$ has a meromorphic extension with a possible hole at $\xi=0$. Due to the estimate for small $\xi$
		\begin{equation*}
			|\xi \mathcal{A}_-(\xi)w(\xi)|\le C |\xi|^{\eps-s}
		\end{equation*} 
		the order of the pole at the origin is weaker than $\xi^{-1}$. Due to Riemann's theorem on removable singularities, see e.g. \cite[Chapter V, Theorem 1.2]{Con78}, the singularity is removable and we obtain a holomorphic function on the full space $\C$. \smallskip
				
		
		We are in the position to finalize the proof. Note that $\tilde{w}(\xi)\coloneq  \xi \mathcal{A}_-(\xi)w(\xi)$ is a holomorphic function with sublinear growth, due to 
		\begin{equation*}
			|\tilde{w}(\xi)|\le C(1+|\xi|)^{1+\eps-s}\le C(1+|\xi|)^{1-s/2}
		\end{equation*}
		where we used \eqref{eq:v-N-estimate} and $\eps<s/2$. Due to Cauchy's inequalities, see e.g. \cite[Corollary 4.3]{StSh03}, we know that 
		\begin{equation*}
			|\tilde{w}'(\xi)|\le \frac{\sup_{z\in \partial B_R(\xi)}|\tilde{w}(z)|}{R}\le C \frac{(R+\abs{\xi})^{1-s/2}}{R}\to 0  
		\end{equation*}
		in the limit $R\to \infty$. Clearly, we find, see e.g. \cite[Corollary 3.4]{StSh03}, a constant $c_1\in \C$ such that $\tilde{w}(\xi)=c$ which, by definition, yields
		\begin{equation*}
			w(\xi)= \frac{c_1}{\xi \mathcal{A}_-(\xi)}.
		\end{equation*}
		By the representation of $\sF u$ we know
		\begin{equation*}
			\sF u = \frac{c_1}{\xi \mathcal{A}_-(\xi)}+ \sum_{k=0}^{\lfloor 2s-\eps \rfloor}a_k\delta_0^{(k)}.
		\end{equation*}
		
		Using similar arguments for the solution $b$ from \autoref{prop:one-dim-sol-halfline}, we find three complex numbers $\tilde{a}_0, \tilde{a}_1, c_2$ such that
		\begin{equation}\label{eq:fourier-b-representation}
			\sF b = \frac{c_2}{\xi \mathcal{A}_-(\xi)}+ \sum_{k=0}^{\lfloor 2s-\eps \rfloor}\tilde{a}_k\delta_0^{(k)}.
		\end{equation}
		Since $b$ growths like $x_+^s$ the constant $c_2$ cannot be zero. Combining these two representations 
		\begin{equation*}
			\sF u = \frac{c_1}{c_2} \sF b + \sum_{k=0}^{\lfloor 2s-\eps \rfloor}\big(a_k- \frac{\tilde{a}_k}{c_2}\big)\delta_0^{(k)}
		\end{equation*}

		Since the Fourier transform is a bijection on the space of tempered distributions, see e.g. \cite[Theorem 2.2.14]{Gra14} and duality, we are left with 
		\begin{equation*}
			u = \frac{c_1}{c_2} b + \sum_{k=0}^{\lfloor 2s-\eps \rfloor}\tilde{c}_k\xi^k
		\end{equation*}
		for two complex numbers $\tilde{c}_0$ and $\tilde{c}_1$. Since $u=b=0$ on $\R_-$, both $\tilde{c}_0$ and $\tilde{c}_1$ need to be zero. 
	\end{proof}

	\section{Boundary estimate via dyadic decomposition}\label{sec:boundary-estimate}
	This section is dedicated to proving that solution to \eqref{eq:localised-equation} decay at the boundary of order $s$. The main results are \autoref{th:boundary-estimate} and \autoref{th:Cs-regularity}. Our method is an adaptation of the ideas in \cite{LaUr88} to the nonlocal setup in combination with the nontrivial half line solution from \autoref{sec:half-line}. The method entails a specific linearly approximation of the boundary $\partial\Omega$ near a boundary point in a dyadic scale, see \autoref{fig:geom-iteration}. This allows us to compare the solution in the curved domain iteratively to a nontrivial solution a carefully attached half space. The construction of such half space solutions is obtained in the following lemma by combining \autoref{lem:basic-properties-Gs} and \autoref{prop:one-dim-sol-halfline}.
	
	\begin{lemma}\label{lem:half-space-solution}
		Let $L\in \Gs$ and $\theta\in S^{d-1}$. We find a continuous distributional and strong solution $b_{H}:\R^d\to \R$ to 
		\begin{equation*}
			\begin{aligned}
				L b_H &= 0 \text{ in } \{ x\cdot \theta>0 \},\\
				b_H &=0 \text{ on }\{ x\cdot \theta\le 0 \}
			\end{aligned}
		\end{equation*}
		and a positive constant $C$ such that $C^{-1}(x\cdot \theta)_+^s\le b_H(x)\le C (x\cdot \theta)_+^s$. The function $b_H$ is constant in any direction orthogonal to $\theta$. 
	\end{lemma}
	\begin{proof}
		By \autoref{lem:basic-properties-Gs}, we find an operator $\tilde{L}$ acting on one-dimensional functions adapted to the half space $\{x\cdot \theta>0\}$. Let $\tilde{b}$ be the function from \autoref{prop:one-dim-sol-halfline} adapted to $\tilde{L}$. We set $b_H(x)\coloneq \tilde{b}(x\cdot \theta)$. Combining \autoref{lem:basic-properties-Gs} and \autoref{prop:one-dim-sol-halfline}, we find for any $x\in \R^d$ such that $x\cdot \theta >0$
		\begin{equation*}
			L b_H(x)= \tilde{L} \tilde{b} (x\cdot \theta)= 0.
		\end{equation*}
		The properties of $\tilde{b}$ yield the remaining statements of the lemma.
	\end{proof}
	
	In the proof of the boundary estimate \autoref{th:boundary-estimate}, we compare a solution to \eqref{eq:localised-equation} with the solution from \autoref{lem:half-space-solution} using the maximum principle on different scales. Since the problem is nonlocal in nature, it is important to cut off the solution on the right scale in order to treat the tails appropriately. This technical ingredient is contained in the following lemma which deals with sub-$L$-harmonic functions near a flat boundary. 
	\begin{lemma}\label{lem:s-harmonic-measure-estimate}
		Let $L\in \Gs$, $0<r<1$, $\theta\in S^{d-1}$, and $D\subset B_r\cap \{x\cdot \theta> -H\}$ for some $H\in \R$. Assume that $w\in L_{2s-\eps}^\infty(\R^d)$ is a lower-semicontinuous distributional (or weak) subsolution to $L w \le 0$ in $D$ and $w\le 0$ on $\overline{B_{2r}}\setminus D$. Then, we find a positive constant $C=C(d,s,\lambda, \Lambda)$ such that for any $x\in \overline{B_{2r}}$
		\begin{equation*}
			w(x)\le C(x\cdot \theta+H)_+^s r^s\sup_{x\in D}\int_{B_{2r}(-x)^c} w_+(x+h) K(\d h).
		\end{equation*}
	\end{lemma}
	\begin{proof}
		Since the class $\Gs$ is invariant under rotation, see \autoref{lem:basic-properties-Gs}, we may assume that $\theta=e_d$. Let $b_H$ be the solution from \autoref{lem:half-space-solution} such that $L b_H=0$ in $\R_+^d$. We also know from the same lemma that $C^{-1}(x_d)_+^s\le b(x)\le C(x_d)_+^s$ for all $x$. 
		
		Note that for $x\in D\subset B_r\cap \{x_d\ge -H\}$
		\begin{align*}
			L [b_H(\cdot + He_d) \1_{B_{4r}}](x)&= 0+ \int_{\R^d} b_H(x+h+He_d)(1-\1_{B_{4r}}(x+h))K(\d h)\\
			&\ge C \int_{B_{5r}(0)^c} (h_d)_+^sK(\d h)\ge  C \sum_{k=3}^{\infty} \int_{B_{2^{k+1}r}(0)\setminus B_{2^kr}(0)} (h_d)_+^sK(\d h)\\
			&\ge  C \sum_{k=3}^{\infty} 2^{-(k+1)(2-s)}r^{s-2}\int_{B_{2^{k+1}r}(0)} (h_d)_+^2K(\d h)\ge C\lambda \sum_{k=2}^{\infty}2^{-sk}r^{-s}.
		\end{align*}
		Here, we used \eqref{X2b}. We have shown that $L[b_H(\cdot+He_d)\1_{B_{4r}}] \ge c_1 r^{-s}$ in $D$. Note that $b_H(\cdot + He_d) \1_{B_{4r}}$ is a weak and continuous distributional solution in $D$, too. Now, we calculate how the tail influences the calculations. Note that for $x\in D\subset B_r$
		\begin{equation*}
			L [w\1{B_{2r}^c}](x)\ge - \sup_{x\in B_r}\int_{B_{2r}(-x)^c}w_+(x+h)K(\d h)\eqcolon -A.
		\end{equation*}
		Let us define the function $W(x)\coloneq  c_1^{-1}Ar^sb_H(x+He_d)\1_{B_{4r}}(x)-w(x)\1_{B_{2r}}(x)$. Note that for $x\in D$ due to the above considerations
		\begin{equation*}
			L W(x)\ge A+ L [w\1_{B_{2r}^c}](x)\ge 0.
		\end{equation*}
		Moreover, on $D^c$ the function $W$ is nonnegative due to the assumption $w\le 0$ on $\overline{B_{2r}}\setminus D$. The maximum principle, see, e.g., \cite[Lemma 2.3.3 or Lemma 2.3.5]{FeRo24}, yields for any $x\in D$
		\begin{equation*}
			w(x)\le c_1^{-1}Ar^sb_H(x+He_d)\1_{B_{4r}}(x)\le C (x_d+H)_+^s r^s \sup_{x\in B_r}\int_{B_{2r}(-x)^c}w_+(x+h)K(\d h).
		\end{equation*}
	\end{proof}
	
	\begin{lemma}\label{lem:barrier-estimate}
		Let $L\in \Gs$, $r>0$, and $\Phi$ be the weak solution to 
		\begin{equation*}
			\begin{split}
				L \Phi&=0 \text{ in }B_{2r}\cap \{x_d>0\},\\
				\Phi &=1 \text{ on }B_{2r}^c\cap \{ x_d>0\},\\
				\Phi &=0 \text{ on }\{ x_d\le 0 \}.
			\end{split}
		\end{equation*}
		We find a constant $C$ such that for any $x\in B_r$
		\begin{equation*}
			\Phi(x)\le C \frac{(x_d)_+^s}{r^s}
		\end{equation*} 
		and for any $x\in B_{2r}$
		\begin{equation*}
			\Phi(x)\ge C^{-1}\frac{(x_d)_+^s}{r^s}.
		\end{equation*}
	\end{lemma}
	\begin{proof}
		Since the class $\Gs$ is invariant under scaling, it suffices to prove the result for $r=1$. By truncating the solution $b(x_d)$, we find a constant $c_1\ge 1$ such that, as in \autoref{lem:s-harmonic-measure-estimate}, for any $x\in B_{2}$ 
		\begin{equation*}
			L[b\1_{B_{4}}](x)\in [c_1^{-1},c_1].
		\end{equation*} 
		Fix a smooth function $\eta$ such that $\eta=1$ on $B_{2}^c$, $\eta=0$ in $B_{3/2}$, and $0\le \eta\le 1$. Then, $|L\eta|\le c_2 $ in $B_{2}$. Thus, the function $\tilde{b}\coloneq  \eta+ (1+c_2)/c_1b\1_{B_{4r}}$ satisfies $L\tilde{b}\ge 0 = L\Phi$ in $B_{2}\cap \{x_d>0\}$. Moreover, clearly $\tilde{b}\ge \eta\ge 1_{B_{2}^c\cap \{x_d>0\}}=\Phi$ on $(B_{2}\cap \{x_d>0\})^c$. The weak maximum principle, see, e.g., \cite[Lemma 2.3.3 or Lemma 2.3.5]{FeRo24}, completes the proof of the first inequality. In order to prove the lower bound, we truncate $\Phi\1_{B_2}$ as well. This yields $L\Phi\1_{B_4}\in [c_3^{-1},c_3]$. Due to \autoref{lem:half-space-solution}, we know that $b\ge c_4(x_d)_+^s$. Defining the function $w\coloneq (\Phi-(1+c_4+c_3)^{-1} c_1 2^{-s} b )\1_{B_2}$ we find 
		\begin{equation*}
			Lw\ge c_3^{-1}-(1+c_4+c_3)^{-1} 2^{-s}\ge 0 \text{ in }B_{2}\cap \{x_d>0\}
		\end{equation*}
		and $w\ge (\1_{\{x_d>0\}}- (1+c_4+c_3)^{-1} c_1 c_4 2^{-s} (x_d)_+^s  )\1_{B_2}\ge 0$ in $(B_{2}\cap \{x_d>0\})^c$. Again, the weak maximum principle, yields the desired estimate.
	\end{proof}
	
	The previous two lemmata allows us to tackle the desired boundary estimate which is essential for the proof of our main result \autoref{thm-intro2}.	
	\begin{theorem}\label{th:boundary-estimate}
		Let $L\in \Gs$ and $\Omega\subset \R^d$ be a $C^{1,\omega}$-domain with modulus $\omega$ satisfying \eqref{eq:s-dini}. For any $f\in L^\infty(\Omega\cap B_2)$, any continuous distributional or weak solution $u\in L_{2s-\eps}^\infty(\R^d)$ to \eqref{eq:localised-equation} satisfies the estimate 
		\begin{equation*}
			|u(x)|\le C |x-z|^s \Big( \|f\|_{L^\infty(\Omega\cap B_2)} + \sup_{x\in B_1}\int_{B_{2}(-x)^c} |u(x+h)|K(\d h) \Big)
		\end{equation*}
		for any $x\in B_1$, any $z\in \partial\Omega\cap B_{3/2}$, and some constant $C=C(\lambda, \Lambda, s,d,\omega, \Omega)$. 
	\end{theorem}	
	The following proof is inspired by \cite{LaUr88}. 
	
	\begin{proof}
		Let $z\in\partial\Omega \cap B_{3/2}$. Since the operator $L$ is translation invariant and the class $\Gs$ is invariant under rotations, we can assume that $z=0$ and that the inner normal vector at the boundary is $e_d$. Moreover, simply dividing by a constant, we may assume that 
		\begin{equation*}
			\|f\|_{L^\infty(\Omega\cap B_2)} + \sup_{x\in B_1}\int_{B_{2}(-x)^c}u_+(x+h)K(\d h)\le 1.
		\end{equation*} 
		
		\textit{Step 1.} Let $r_0$ be the localization radius of the domain and define dyadic radii $r_k\coloneq 2^{-k}r_0$ and localizations of the domain $D_k\coloneq  \Omega\cap B_{r_k}$. The geometric deviation of $\partial\Omega$ from the plane $\{x_d=0\}$ is bounded by $\delta_k\coloneq C\omega(r_k)r_k$ due to the regularity of the boundary of $\Omega$. We define shifts $p_k\coloneq c_1 \sum_{n=k}^{\infty}\delta_n$. The positive constant $c_1\ge 1$ will be chosen later. Note that the sum is finite. This choice guarantees the set inclusions
		\begin{equation*}
			D_{k+1}\subset D_{k}\subset H_{k}\coloneq \{x\in \R^d\mid x_d> -p_{k} \}\subset H_{k-1}.
		\end{equation*}  
		Our goal is to compare the solution $u$ to the nontrivial solution $b_H(x+p_k e_d)=b(x_d+p_k)$ in the half space $H_{k}$, where $b_H$ is the function from \autoref{lem:half-space-solution} satisfying $L b_H(x+p_ke_d )=0$ for $x\in H_k$ and $b_H(x+p_ke_d)=0$ else. 
		
		\textit{Step 2.} We want to inductively construct a nondecreasing sequence of number $M_k\in (0,\infty)$ such that \begin{equation}\label{eq:M-k-estimate}
			u(x)\le M_k b(x_d+p_k)\text{ for all }x\in D_{k+1}.
		\end{equation} 
		
		Note that the first two constants $M_0, M_1$ can simply be chosen using the fact that $u$ is bounded, i.e., 
		\begin{equation*}
			M_0\coloneq  \sup_{x\in D_1} C\frac{|u(x)|}{(p_{0}-p_{1})^s},\quad M_1\coloneq \max\{M_0, \sup_{x\in D_2} C\frac{|u(x)|}{(p_{1}-p_{2})^s}\}.
		\end{equation*}
		
		Now, we proceed inductively and assume that we have proven the estimate 
		\begin{equation*}
			u(x)\le M_j b(x_d+p_j) \text{ for all }x\in D_{j+1}
		\end{equation*}
		for all $j\in \{1,\dots, k\}$ and that $M_j$ is nondecreasing. We will prove the same estimate for $j=k+1$ in this step. We define the error function 
		\begin{equation*}
			w(x) \coloneq  u(x)- M_k b(x_d+p_{k+1})-M_k [b]_{C^s(\R)} (p_{k-1}-p_{k+1})^s\Phi_k(x)-\Psi_k(x),
		\end{equation*}
		where $\Psi_k$ is the solution to 
		\begin{equation*}
			\begin{split}
				L \Psi_k&=1 \text{ in }B_{r_{k+1}}\cap H_{k+1},\\
				\Psi_k&=0 \text{ on }\big(B_{r_{k+1}}\cap H_{k+1}\big)^c
			\end{split}
		\end{equation*}
		and $\Phi_k(x)$ is the solution to 
		\begin{equation*}
			\begin{split}
				L \Phi_k&= 0 \text{ in }B_{r_{k+1}}\cap H_{k+1},\\
				\Phi_k&= 0 \text{ on }H_{k+1}^c,\\
				\Phi_k&= 1 \text{ on }B_{r_{k+1}}^c\cap H_{k+1}.
			\end{split}
		\end{equation*}
		
		Clearly, this function solves $L w= f-1 \le 0$ in $D_{k+1}$. In order to apply \autoref{lem:s-harmonic-measure-estimate}, we need to ensure that $w\le 0$ in $B_{2r_{k+1}}\setminus D_{k+1}$. If $x\notin \Omega$, then this follows directly from $u(x)=0$ and the nonnegativity of all other functions involved. If instead $x\in \Omega\cap B_{2r_{k+1}}\setminus B_{r_{k+1}}$, then 
		\begin{align*}
			w(x)\le u(x)-M_kb(x_d+p_{k-1})+M_k\Big( b(x_d+p_{k-1})-b(x_d+p_{k+1}) -[b]_{C^s(\R)} (p_{k-1}-p_{k+1})^s \Big).
		\end{align*}
		Since $2r_{k+1}= r_{k}$ and $u(x)\le M_kb(x_d+p_{k-1})$ for $x\in D_k$, we proceed using the $C^s$-regularity of the function $b$: 
		\begin{equation*}
			w(x)\le M_k\Big( b(x_d+p_{k-1})-b(x_d+p_{k+1}) - [b]_{C^s(\R)} (p_{k-1}-p_{k+1})^s \Big) \le 0
		\end{equation*}
		This allows us to apply \autoref{lem:s-harmonic-measure-estimate} which yields for any $x\in B_{r_{k+1}}$
		\begin{equation}\label{eq:excess-estimate}
			w(x)\le Cb(x_d+p_{k+1})r_{k+1}^s\sup_{x\in D_{k+1}}\int_{B_{2r_{k+1}}(-x)^c}w_+(x+h)K(\d h).
		\end{equation}
		The next step comprises of bounding this excess tail term. 
		
		\textit{Step 3.} We set $\tilde{w}(x) \coloneq  u(x)- M_k b(x_d+p_{k+1})$. Since $\Phi_k$ and $\Psi_k$ are nonnegative, we know that $w(x)_+\le \tilde{w}(x)_+$. We define the tail error 
		\begin{equation*}
			T_k\coloneq \sup_{x\in D_{k+1}}\int_{B_{r_{k}}(-x)^c}\tilde{w}(x+h)_+K(\d h)
		\end{equation*}		
		and the dyadic rings $A_j\coloneq  B_{r_{j}}\setminus B_{r_{j+1}}$. Now, we decompose the tail dyadically 
		\begin{equation*}
			T_k\le \sup_{x\in D_{k+1}}\int_{B_{r_1}(-x)^c}\tilde{w}(x+h)_+ K(\d h)+\sum_{j=0}^{k-2} \int \1_{A_{j+1}}(x+h)\tilde{w}(x+h)_+ K(\d h).
		\end{equation*}
		Clearly, the first term is easily bounded by a constant, independent of $k$, using an $L^\infty$-estimate for $u$, see \cite[Proposition 2.3.11]{FeRo24}. By the induction, we know for any $j\in \{0,\dots, k-2\}$ and $x\in D_{k+2}$ that 
		\begin{align*}
			\I{j}&\coloneq \int \1_{A_{j+1}}(x+h)\tilde{w}_+(x+h)K(\d h)\\
			&\le \int \1_{A_{j+1}}(x+h)\big(M_{j} b(x_d+h_d+p_j)-M_k b(x_d+h_d+p_{k+1}) \big)_+K(\d h).
		\end{align*}
		Since the sequence $M_j$ is nondecreasing and $b\in C^s(\R)$, we estimate this from above by 
		\begin{equation*}
			\begin{split}
				\I{j}&\le [b]_{C^s} M_k \int \1_{A_j}(x+h)(p_j-p_{k+1})_+^sK(\d h)\\
				&\le  C [b]_{C^s} M_k\Big(\sum_{n=j}^{k}\delta_n\Big)^sK(B_{r_{j+2}}^c)\le  C M_k \sum_{n=j}^{k}r_n^s\omega(r_n)^sr_j^{-2s}. 
			\end{split}
		\end{equation*}
		Since $r_k$ are dyadic, summing this over all possible $j$ yields
		\begin{equation*}
			\sum_{j=0}^{k-2}\I{j}\le C M_k \sum_{n=0}^{k-2}r_n^s\omega(r_n)^s \sum\limits_{j=0}^{n} r_j^{-2s}\le C M_k \sum_{n=0}^{k}r_n^{-s}\omega(r_n)^s.
		\end{equation*}
		
		\textit{Step 4.} As in the proof of \autoref{lem:s-harmonic-measure-estimate}, using a truncation of $b(x_d+p_{k+1})$ and the maximum principle yields 
		\begin{equation*}
			\Psi_k(x)\le C r_{k+1}^s(x_d+p_{k+1})_+^s \text{ for all }x\in \R^d.			
		\end{equation*}
		By \autoref{lem:barrier-estimate},
		\begin{equation*}
			\Phi_k(x)\le C \frac{(x_d+p_{k+1})^s_+}{r_{k+1}^s}\text{ for all }x\in B_{r_{k+2}}.
		\end{equation*}
		
		\textit{Step 5.} Due to \eqref{eq:excess-estimate} and the tail estimates from step 3 and step 4, we find for any $x\in D_{k+2}$
		\begin{equation}
			u(x)\le \Big(M_k+C r_k^s+M_kC \frac{(p_{k-1}-p_{k+1})^s}{r_{k+1}^s} +CM_k \sum_{n=0}^{k}2^{-s(k-n)}\omega(r_n)^s\Big)b(x_d+p_{k+1}).
		\end{equation}
		This finally allows us to define the next constant in the induction proof, i.e., $M_{k+1}$, and the estimate \eqref{eq:M-k-estimate} follows. 
		
		\textit{Step 6.} Now, we aim to prove that the constants $M_{k}$ are bounded independently of $k$. By the relation $M_{k+1}= M_k\big(1+ C N_k\big)+C r_k^s$ where $N_k\coloneq  \sum_{n=0}^{k}2^{-s(k-n)}\omega(r_n)^s$, the discrete Grönwall's inequality yields
		\begin{align*}
			M_k &\le  M_0\exp\Big(C\sum_{j=0}^{k-1}N_j\Big) +C\sum_{i=0}^{k-1} r_i^s \exp\Big(C\sum_{j=i+1}^{k-1} N_j\Big)\\
			&\le \Big( M_0 + C\sum_{i=0}^{\infty} r_0^s2^{-si} \Big)\exp\Big(C\sum_{j=0}^{\infty}N_j\Big).
		\end{align*}
		Note that $\sum_{j=0}^{\infty}N_j$ is finite since
		\begin{align*}
			\sum_{j=0}^{\infty}N_j= \sum_{j=0}^{\infty} \sum_{n=0}^{j}2^{-s(j-n)}\omega(r_n)^s= \sum_{n=0}^{\infty}\omega(r_0 2^{-n})^s \sum_{j=n}^{\infty}2^{-s(j-n)}\le \frac{C}{s} \int_0^{r_0}\frac{\omega(t)^s}{t}\d t.
		\end{align*}
		This proves that $M_k$ is uniformly bounded by a constant $C$ independent of $k$. \smallskip
		
		This allows us to complete the proof. If $|x-z|\ge r_0$, then the proof trivially follows from an $L^\infty$-estimate, see \cite[Proposition 2.3.11]{FeRo24}. Let $x\in B_{r_0}(0)\cap \Omega$ and fix $k\in \N$ such that $2^{-k-2}r_0\le |x|\le 2^{-k-1}r_0$. Using the inequality \eqref{eq:M-k-estimate}, we know 
		\begin{align*}
			u(x)&\le M_k b(x_d+p_k)\le C \big(x_d+ Cr_0\sum_{n=k}^{\infty} \omega(r_02^{-n})2^{-n} \big)^s\le C \big(x_d+ C 2^{-k} r_0 \big)^s\\
			&\le C (x_d+|x|)^s\le C |x-z|^s.
		\end{align*}
		Repeating the argument with $-u$ instead of $u$ yields the result.
	\end{proof}

	Having the boundary estimate \autoref{th:boundary-estimate} at hand, allows us to prove a generalized version of the boundary regularity result in \autoref{thm-intro2} using standard tools. 
	\begin{theorem}\label{th:Cs-regularity}
		Let $L\in \Gs$ and $\Omega\subset \R^d$ be an open and bounded set that satisfies the exterior $C^{1,\omega}$-paraboloid property with a uniform radius at every boundary point $z\in \partial\Omega \cap B_{2}$ with a modulus of continuity $\omega$ that satisfies \eqref{eq:s-dini}. For any $f\in L^\infty(\Omega\cap B_2)$, any weak solution $u\in L_{2s-\eps}^\infty(\R^d)$ to \eqref{eq:localised-equation} belongs to $C^s(\overline{B_{1}})$ and satisfies the estimate
		\begin{equation*}
			\norm{u}_{C^s(B_1)}\le C\big(\norm{f}_{L^\infty(B_2)}+ \norm{u}_{L^\infty_{2s-\eps}(\R^d)} \big).
		\end{equation*}
		The constant $C$ depends on $s,\eps,\lambda,\Lambda$, and $\omega$. 
	\end{theorem}	
	\begin{proof}
		Let $z\in \partial\Omega \cap B_2$ be arbitrary. Using the paraboloid from \autoref{def:paraboloid}, we construct a $C^{1}$-domain $\Omega_\omega\subset B_{2}(z)$ with modulus $2\omega$ that contains $\Omega\cap B_{1}(z)$ and touches $\partial\Omega$ only in $z$. Let $v$ be the solution to 
		\begin{equation*}
			\begin{aligned}
				Lv&=  \|f\|_{L^\infty(\Omega\cap B_2)}  &\text{ in }\Omega_\omega,\\
				v&= u &\text{ on }\Omega_\omega^c.
			\end{aligned}
		\end{equation*}		
		An application of \autoref{th:boundary-estimate} yields for any $x\in B_{1}(z)$
		\begin{align*}
			|v(x)|&\le C |x-z|^s \Big( \|f\|_{L^\infty(\Omega\cap B_2)} + \sup_{x\in B_1(z)}\int_{B_{2}(-x)^c}|u(x+h)|K(\d h) \Big).
		\end{align*}
		The weak maximum principle yields $u\le v$ in $\Omega_\omega$.
		Since $z$ was arbitrary, the previous estimate holds for $d_{\Omega}(x)^s$ in place of $|x-z|^s$. The same estimate holds for $-u$ using the same ideas. Now, we simply combine a rescaled version of the interior regularity estimate \cite[Theorem 2.4.3]{FeRo24} with this boundary decay estimate.
	\end{proof}
	
	\begin{remark}
		The boundary regularity result \autoref{th:Cs-regularity} can be directly applied in order to obtain Hölder regularity of solution to the inhomogeneous Dirichlet-problem, see \cite[Section 2.6.7]{FeRo24} and \cite{Gru25}.
	\end{remark}
	
	\section{Hopf-type lemma}\label{sec:hopf}
	The flip side of \autoref{th:boundary-estimate} is the Hopf-boundary lemma. It allows us to describe the minimal decay of a positive solution to \eqref{eq:localised-equation} at the boundary. Although interesting on its own, it will be crucial in the proof of the boundary Harnack principle \autoref{th:boundary-harnack} and, more precisely, in the expansion near a boundary point, see \autoref{th:expansion}. The proof of this Hopf-type result follows roughly the same strategy as \autoref{th:boundary-estimate} but, instead of attaching half spaces outside of $\Omega$ on dyadic scales, we approximate the boundary $\partial\Omega$ near a boundary point from inside of the domain. This, again, allows us to compare a positive solution to the half space solutions from \autoref{lem:half-space-solution}. 
	
	\begin{theorem}\label{prop:hopf}
		Let $L\in \Gs$ and $\Omega\subset \R^d$ be a $C^{1,\omega}$-domain with modulus $\omega$ satisfying \eqref{eq:s-dini}. For any nonnegative, upper-semicontinuous distributional or weak super-solution $u\in L_{2s-\eps}^\infty(\R^d)$ to $L u \ge 0$ in $\Omega\cap B_2$, $u>0$ in $\Omega\cap B_2$, and $u\ge 0$ on $\Omega^c\cap B_2$, we find a constant $C$ such that $u(x)\ge Cd_{\Omega}(x)^s$ for all $x\in \Omega\cap B_1$. 
	\end{theorem}
	\begin{proof}
		The setup to this proof is in line with \autoref{th:boundary-estimate}. Again, let $r_0, r_k,\delta_k, p_k$, and $b$ be as in \autoref{th:boundary-estimate}. We assume that $z\in \partial\Omega$ is simply $z=0$ by translation invariance and define $H_k\coloneq \{x\in \R^d\mid x_d>p_k \}$ and $D_k\coloneq B_{r_k}\cap H_k$. Clearly, $D_k\subset \Omega\cap B_{r_k}$. 
		
		\textit{Step 1.} We aim to construct a nonincreasing sequence of positive numbers $m_k$ such that for all natural $j\le k$
		\begin{equation}\label{eq:induction-hopf}
			u(x)\ge m_j b(x_d-p_{j}) \text{ for all } x\in B_{r_{j}}.
		\end{equation}
		
		We will construct the sequence $m_k$ inductively and set the initial value later. Let's assume that \eqref{eq:induction-hopf} is true for all $1\le j\le k$. Next, we construct $m_{k+1}$. Let us define the error function for $k+1$
		\begin{equation*}
			w(x)\coloneq  u(x)-m_k b(x_d-p_{k+1})+ m_k[b]_{C^s(\R)}(p_{k}-p_{k+1})^s\Phi_k.
		\end{equation*}
		Here, $\Phi_k$ is the solution to the problem $L\Phi_k=0$ in $D_{k+1}$ and $\Phi_k=\1_{H_{k+1}}$ on $D_{k+1}^c$. By construction, the function $w$ is a super-solution in $D_{k+1}$. Our goal is to apply \autoref{lem:s-harmonic-measure-estimate} to $-w$ in $D_{k+1}$. In order to do so, we need to check that $w\ge 0$ on $\overline{B_{r_{k}}}\setminus D_{k+1}$. This is a direct consequence of \eqref{eq:induction-hopf} from the previous step $j=k$, i.e., $u(x)\ge m_k b(x_d-p_{k})$ for $x\in B_{r_k}$. This yields for $x\in \overline{B_{r_k}}\setminus D_{k+1}$
		\begin{equation*}
			w(x)\ge m_k b(x_d-p_{k})-m_k b(x_d-p_{k+1})+ m_k[b]_{C^s(\R)}(p_{k}-p_{k+1})^s\1_{H_{k+1}}(x)\ge 0
		\end{equation*}
		This allows us to apply \autoref{lem:s-harmonic-measure-estimate} to the function $-w$. This yields for any $x\in D_{k+1}$
		\begin{equation*}
			w(x)\ge -C (x_d-p_{k+1})^sr_{k+1}^s \sup_{x\in D_{k+1}} \int_{B_{r_{k}}(-x)^c} w_-(x+h)K(\d h).
		\end{equation*}
		
		\textit{Step 2.} We need to estimate the nonlocal tail. Clearly, the term 
		\begin{equation*}
			\sup_{x\in D_{k+1}} \int_{B_{r_{1}}(-x)^c} w_-(x+h)K(\d h)
		\end{equation*}
		is bounded by a constant. Again, we define the annuli $A_j\coloneq  B_{r_j}\setminus B_{r_{j+1}}$. Due to the inductive assumption and since $m_j$ is nonincreasing, we know for any $j\in \{1,\dots, k\}$ and $x\in D_{k+1}$
		\begin{align*}
			\int \1_{A_j}(x+h) w(x+h)_- K(\d h)&\le \int \1_{A_j}(x+h) \big( m_k b(x_d-p_{k+1})-m_jb(x_d-p_j)\big)_+ K(\d h)\\
			&\le m_k[b]_{C^s} \int \1_{A_j}(x+h) \big( p_j- p_{k+1}\big)^s K(\d h)\\
			&\le C m_k r_j^{-2s}\sum_{n=j}^{k}\omega(r_n)^sr_n^s.
		\end{align*}
		Summing this over all possible $j\in \{1,\dots, k\}$ yields for $x\in D_{k+1}$
		\begin{equation*}
			w(x)\ge - m_k C b(x_d-p_{k+1}) \sum_{n=0}^{k}2^{-(k-n)s}\omega(r_n)^s.
		\end{equation*}
		By the definition of $w$ and using the lower bound for $\Phi_k$ from \autoref{lem:barrier-estimate}, we find for $x\in D_{k+1}$
		\begin{align*}
			u(x)\ge m_k b(x_d-p_{k+1})\Big(1 - C\sum_{n=0}^{k}2^{-(k-n)s}\omega(r_n)^s\Big).
		\end{align*}
		This proves the inequality \eqref{eq:induction-hopf} for $j=k+1$ with the choice
		\begin{equation}\label{eq:m_k}
			m_{k+1}\coloneq  m_k\big(1 - C N_k\big), \qquad N_k\coloneq  \sum_{n=0}^{k}2^{-(k-n)s}\omega(r_n)^s.
		\end{equation}
		
		\textit{Step 3.} We need to prove that the nonincreasing sequence $m_k$ stays bounded away from $0$ for large $k$. By the relation \eqref{eq:m_k}, we know that for any $k\ge K\in \N$
		\begin{equation}\label{eq:m-k-at-K}
			m_{k}= m_K \prod_{n=K}^{k-1} \big(1 - C N_n\big).
		\end{equation}
		Note that the sequence $N_k$ is summable since
		\begin{equation*}
			\sum_{k=0}^{\infty}N_k = \sum_{n=0}^{\infty}\omega(r_n)^s \sum_{k=n}^{\infty}2^{-(k-n)s}\le C \int_{0}^{2r_0} \frac{\omega(t)^s}{t}\d t.
		\end{equation*}
		Thus, the sequence $N_k$ converges to zero. Let's pick $K\in \N$ such that $\big(1 - C N_k\big)$ is larger than $1/2$ for all $k\ge K$. This finally, allows us to pick the initial value $m_K$ to kick off the induction. Since $u$ is positive in $\Omega\supset\supset D_K$, the number $m_K\coloneq  \inf_{x\in D_{K}} u(x)/b(x_d-p_K)$ is nonzero. 
		
		The product in \eqref{eq:m-k-at-K} converges to a nonzero constant, if the logarithm of the product stays bounded away from $-\infty$. We calculate
		\begin{equation*}
			\ln \big(\prod_{n=K}^{k-1} \big(1 - C N_n\big)\big)= \sum_{n=K}^{k-1}\ln(1-CN_n).
		\end{equation*}
		Since $\ln(1-t)\ge -2t$ for $t\in (0,1/2)$, we estimate
		\begin{equation*}
				\ln \Big(\prod_{n=K}^{k-1} \big(1 - C N_n\big)\Big)\ge -C\sum_{n=K}^{k-1} N_n\ge - C \sum_{n=K}^{\infty}N_n>-\infty
		\end{equation*}
		In the last inequality, we used that $N_n$ is summable, as proven above. 
		
		\textit{Step 4.} This allows us to finalize the proof. Note that the choice of the boundary point $0\in \partial \Omega$ was arbitrary and the estimate \eqref{eq:m_k} also holds when translated and rotated to another boundary point $z\in \partial\Omega$. Let $x\in \Omega$. If $d_{\Omega}(x)>r_K/2$, then the estimate is trivial due to $u$ being positive in $\Omega$. Else, we let $z$ be the projection of $x$ onto the boundary $\partial\Omega$. Without loss of generality, we assume that the projection is $0\in \partial\Omega$ and let's assume that the domain is rotated such that the normal at $z$ is $e_d$. We fix $k\in \N$ such that $r_0 2^{-k-2}\le d_{\Omega}(x)\le r_0 2^{-k-1}$. By \eqref{eq:m_k}, we know
		\begin{equation*}
			u(x)\ge m_k b(x_d-p_{k}) \ge C m_{\infty} (x_d-p_k)^s\ge C m_{\infty}\Big(d_{\Omega}(x)-C \omega(r_k)r_k \Big)^s \ge C m_\infty d_{\Omega}(x)^s.
		\end{equation*}
		Here, we used that $\omega(r_k)$ is small. 		
	\end{proof}
	
	Let us formulate a generalized version of the Hopf result from \autoref{thm-intro2}.  
	\begin{theorem}\label{th:hopf}
		Let $L\in \Gs$, and $\Omega\subset \R^d$ be an open set that satisfies the interior $C^{1,\omega}$-paraboloid property with a uniform radius $r<2$ at every boundary point $z\in \partial\Omega\cap B_2$ with modulus $\omega$ that satisfies \eqref{eq:s-dini}. For any nonnegative, weak super-solution $u\in L_{2s-\eps}^\infty(\R^d)$ to 
		\begin{equation*}
			\begin{aligned}
				L u &\ge 0\text{ in }\Omega\cap B_2,\\
				u&>0\text{ in }\Omega\cap B_2,\\
				u&=0\text{ on }\Omega^c\cap B_2,
			\end{aligned}
		\end{equation*}
		we find a constant $C$ such that $u(x)\ge Cd_{\Omega}(x)^s$ for all $x\in \Omega\cap B_1$. 
	\end{theorem}
		
	\begin{proof}
		Let $x_0\in \Omega\cap B_1$ be arbitrary. We fix $z\in \partial\Omega \cap B_{2}$ and a $C^{1,\omega}$-paraboloid $\mathcal{C}_z$ with apex at $z$ such that $x_0\in \mathcal{C}_z$ and $d_{\mathcal{C}_z}(x_0)\le C d_{\Omega}(x_0)$. By \autoref{prop:hopf}, we know that $u(x)\ge C d_{\mathcal{C}_z}(x)^s$ in $\mathcal{C}_z\cap B_{1}(z)$. Since $d_{\mathcal{C}_z}(x_0)\le Cd_{\Omega}(x_0)$, the result follows. 
	\end{proof}
	
	\begin{remark}
		Instead of weak solution, \autoref{th:Cs-regularity} respectively \autoref{th:hopf} also hold for continuous distributional solutions respectively upper-semicontinuous distributional supersolutions. 
	\end{remark}
	
	\begin{proof}[{Proof of \autoref{thm-intro2}}]
		Combine \autoref{th:Cs-regularity} and \autoref{th:hopf}. 
	\end{proof}

	\section{Liouville in the half space}\label{sec:liouville}
	This section is dedicated to proving a higher order Liouville theorem in the half space. Although interesting on its own, it is a key ingredient in the proof of \autoref{th:expansion}. In order to prove \autoref{th:expansion}, we employ a standard method, i.e., a blow up / contradiction-compactness argument, see, e.g., \cite{RS14}, \cite{RS16}, \cite{RoWe24}, or \cite{SvWe25}. Such arguments yield solutions to \eqref{eq:problem-half-space} after the blow up which need to be classified in order to find a contradiction. 
	\begin{proposition}\label{prop:liouville-half-space}
		Let $L\in \mathcal{G}_s(\lambda, \Lambda)$ and $u:\R^d\to \R$ be a continuous distributional solution to 
		\begin{equation}\label{eq:problem-half-space}
			\begin{split}
				Lu &=0 \text{ in }\R_+^d,\\
				u&=0 \text{ on }(\R_+^d)^c
			\end{split}
		\end{equation}
		satisfying the growth bound $|u(x)|\le C (1+|x|)^{2s-\eps}$ for all $x\in \R^d$ and some $\eps>0$. Then, we find a constant $C$ such that $u(x)= Cb(x_d)$ where $b$ is the function from \autoref{lem:half-space-solution}. If $\eps>s$, then $C=0$. 
	\end{proposition}
	
	In order to prove this Liouville theorem, we use standard methods and combine the one-dimensional version \autoref{prop:liouville} with the Hölder boundary regularity result \autoref{th:Cs-regularity}. The proof is very close to \cite[Theorem 6.2]{RoWe24}. 
	
	\begin{proof}[{Proof of \autoref{prop:liouville-half-space}}]
		The proof is done in two steps. In the first step, we prove that the solution $u$ is constant in tangential directions.
		
		\textit{Step 1.} Let $R>0$, $0<h<R/4$ and $\theta\in S^{d-1}$ be a vector parallel to the half space boundary, i.e., such that $\theta_d=0$. We define $v_h\coloneq   (u(x+h\theta)-u(x))/h$ and $u_R(x)\coloneq   u(Rx)$. Due to \autoref{th:Cs-regularity} we find $\eps>0$ such that
		\begin{equation*}
			\begin{split}
				\norm{v_h}_{L^\infty(B_{R/4})}&\le [u]_{C^{\eps}(B_{R/4+h})}\le R^{-\eps}[u_R]_{C^{\eps}(B_{1/2})}\\
				&\le C R^{-\eps}  \norm{u_R}_{L_{2s-\eps}^\infty}(\R^d)\le C R^{2s-2\eps}.
			\end{split}
		\end{equation*}
		Now, we repeat this process but instead of taking differences of $u$ we take differences of $v_h$. This yields after $k$-repetitions 
		\begin{equation*}
			\norm{\tau_k(u,h)}_{L^\infty(B_{R/4})}\le C R^{2s-(1+k)\eps}.
		\end{equation*}
		Here, $\tau_1(u,h)\coloneq  v_h$ and $\tau_{n+1}\coloneq  \tau_{1}(\tau_n(u,h),h)$. 
		If we choose $k$ sufficiently large, then this term converges to Zero as $R\to \infty$. Since $\theta$ was arbitrary, $u$ is a polynomial in the first $(d-1)$-coordinates. Due to the growth bound, we know that the polynomial is at most of order 1. Thus, we find $f:\R\to \R^{d-1}$ and $g:\R \to \R$ such that 
		\begin{equation*}
			u(x)= f(x_d)\cdot x'+g(x_d)
		\end{equation*}
		for all $x=(x',x_d)\in \R^d$. \smallskip
		
		\textit{Claim.} We claim that $f(x_d)=0$ for all $x_d\in (0,\infty)$.
		
		If $2s\le 1$, then the growth bound already yields the claim. Let's assume $2s>1$. By the interior regularity \cite[Proposition 2.4.4]{FeRo24}, we know that $u$ is differentiable in $\R_+^d$ and $f_i(x_d)=\partial_i u(x)$ for all $i\in \{1,\dots, d-1\}$. Moreover, the estimate \eqref{eq:int_reg_for_u} yields $\partial_i u\in L^1_{\loc}(\overline{\R_+^d})$. By the linearity of $L$, the function $\partial_i u$ is a distributional solution to $L\partial_i u =0$ in $\R_+^d$. Note that $g(x_d)=u(0,x_d)$ is continuous and, thus, $\partial_i u(x)=f_i(x_d)$ is continuous, too. We employ \autoref{lem:basic-properties-Gs} coupled with the one-dimensional Liouville theorem \autoref{prop:liouville}. This yields $\partial_i u(x)= c_i b(x_d)$, where $b$ is the one-dimensional function from \autoref{prop:one-dim-sol-halfline}, see \autoref{lem:half-space-solution}. This yields
		\begin{equation*}
			u(x)-u((x_1,\dots, x_{i-1},0,x_{i+1},\dots, x_d))= c_i b(x_d)x_i.
		\end{equation*}
		Since $b$ behaves like $(x_d)_+^s$, the growth bound on $u$ yields $c_i=0$ for all $i\in \{1,\dots, d-1\}$. This proves the claim and completes the first step. \smallskip
		
		\textit{Step 2.} We know that $u(x)=u((0,x_d))$ and due to \autoref{lem:basic-properties-Gs}, we can employ the one-dimensional Liouville theorem \autoref{prop:liouville}. This yields $u(x)=c b(x_d)$. Clearly, $c=0$ if $\eps>s$ since $u$ grows slower than $(1+|x|)^{s}$ but $b$ behaves like $(x_d)_+^s$ by \autoref{lem:half-space-solution}.
	\end{proof}

	\section{Boundary Harnack principle}\label{sec:bhp}	
	This section is dedicated to the proof of our second main result, the boundary Harnack principle \autoref{th:boundary-harnack}. The next theorem provides an expansion of a solution around a boundary point. This is a main ingredient in the proof of the boundary Harnack principle. 	
	\begin{theorem}\label{th:expansion}
		Let $L\in \Gs$, $0<\eps<s/2$, $\alpha\in (0,s)$, and let $\Omega\subset \R^d$ be a bounded $C^{1,\omega}$-domain with modulus $\omega$ that satisfies \eqref{eq:s-dini}. Moreover, we fix the function $g:\R^d\to \R$ that solves
		\begin{equation}\label{eq:comparison-fct}
			\begin{split}
				Lg 	&= 0\phantom{\1_{\Omega\setminus A}} \text{ in }B_2\cap \Omega,\\
				g	&= \1_{B_2^c}\phantom{0} \text{ on }(B_2\cap \Omega)^c.
			\end{split}
		\end{equation}		
		 Let $f\in L^\infty(\Omega\cap B_1)$ and $u\in L_{2s-\eps}^\infty(\R^d)$ be a continuous distributional or weak solution of the problem 
		\begin{equation*}
			\begin{split}
				Lu	&=f \text{ in }\Omega\cap B_1,\\
				u	&=0 \text{ on }\Omega^c\cap B_1.
			\end{split}
		\end{equation*}
		Then, we find a constant $C$ such that for any boundary point $z\in B_{1/2}\cap \partial\Omega$ we find a number $q_z\in \R$ bounded by $|q_z|\le C(\|u\|_{L_{2s-\eps}^\infty(\R^d)}  +\norm{f}_{L^\infty(\Omega \cap B_1)} )$ such that 
		\begin{equation*}
			|u(x)-q_z g(x)|\le C |x-z|^{s+\alpha} \big( \|u\|_{L_{2s-\eps}^\infty(\R^d)}  +\norm{f}_{L^\infty( \Omega \cap B_1)} \big)
		\end{equation*}
		for all $x\in \overline{B_{1/2}}$.
	\end{theorem}
	
	The proof of \autoref{th:expansion} follows the contradiction compactness arguments as in \cite{RS16}, \cite{RoWe24}, or \cite{SvWe25}. Since we already proved an optimal boundary decay result and a Hopf lemma, see \autoref{th:boundary-estimate} and \autoref{prop:hopf}, we are in the position to prove this expansion in $C^{1}$-domains with modulus $\omega$ that satisfies \eqref{eq:s-dini}. We only need to adapt the proof contained in \cite{SvWe25} slightly to fit our setup. 
	
	\begin{proof}
		Let us assume the contrary. Then we find a sequence of $C^{1,\omega_n}$-domains $\Omega_n$ with moduli $\omega_n$ that satisfy the property \eqref{eq:s-dini} uniformly, i.e., we find a constant $C$ such that for all $n$
		\begin{equation*}
			\int_0^1 \frac{\omega_n(t)^s}{t}\d t\le C. 
		\end{equation*} 
		Additionally, we find inhomogeneities $f_n$ and continuous solutions $u_n\in L_{2s-\eps}^\infty(\R^d)$ such that 
		\begin{equation*}
			\norm{f_n}_{L^\infty(\Omega_n\cap B_1)}+\norm{u_n}_{L_{2s-\eps}^\infty(\R^d)}=1
		\end{equation*}
		that solve
		\begin{equation*}
			\begin{split}
				L_nu_n&=f_n\phantom{0} \text{ in }\Omega_n\cap B_1,\\
				u_n&=0\phantom{f_n} \text{ on }\Omega_n^c\cap B_1
			\end{split}
		\end{equation*}
		and solutions $g_n$ to \eqref{eq:comparison-fct} in the domain $\Omega_n$. But, as in \cite{RS16}, 
		\begin{equation}\label{eq:infinite-term}
			\sup_{n\in \N} \sup_{z\in \partial\Omega_n\cap B_{1/2}}\sup_{0<r<1} \inf_{q\in \R} r^{-s-\alpha} \| u_n - q g_n \|_{L^\infty(B_r(z))}=\infty.
		\end{equation}
		Let $q_{n,z,r}$ be the $L^2(B_r(z))$-projection of $u_n$ over $g_n$, that is 
		\begin{equation*}
			q_{n,z,r}\coloneq   \frac{\int_{B_r(z)} u_n(x) g_n(x)\d x }{\norm{g_n}_{L^2(B_r(z))}^2}. 
		\end{equation*}
		
		We define the auxiliary function $\theta:[0,1] \to [0,\infty]$ via
		\begin{equation*}
			\theta(r)\coloneq  \sup_{n\in \N} \sup_{z\in \partial\Omega_n\cap B_{1/2}}\sup_{r<t<1} t^{-s-\alpha}\| u_n - q_{n,z,r} g_n \|_{L^\infty(B_t(z))}.
		\end{equation*}
		Clearly, $\theta$ is nonincreasing, finite on $(0,1)$, and $\theta(0)=\infty$. As in \cite{RS16} or \cite[Proof of Theorem 6.9]{RoWe24}, we find a sequence $r_m\to 0$ monotonically as $m\to \infty$, a subsequence $\{n_m\}$ such that $L_{n_m}\phi\to L_\infty\phi$ as $m\to\infty$ for any $\phi\in C_c^2(\R^d)$ and some $L_\infty\in \Gs$, and a convergent sequence of boundary points $z_m\in \partial\Omega_{n_m}\cap B_{1/2}$ with limit $z\in \overline{B_{1/2}}$ such that 
		\begin{equation}\label{eq:construction-sequence}
			r_{m}^{-s-\alpha}\| u_{n_m} - q_{n_m,z_m,r_m} g_{n_m}(x)\|_{L^\infty(B_{r_m}(z_m))}\ge \frac{\theta(r_m)}{2}
		\end{equation}
		for all $m\in \N$. Let us ease the notation and use the original sequence for the new subsequence, i.e., $m=n_m$, and write $q_m= q_{m,z_m,r_m}$. 
		
		Now, we are in the position to define our blow up sequence $v_m:\R^d\to \R$ by 
		\begin{equation*}
			v_m(x)\coloneq   \frac{u_{m}(z_m+r_mx)- q_m g_{m}(z_m+r_mx)}{r_m^{s+\alpha}\theta(r_m)}.
		\end{equation*}
		By the choice of $q_m$, we know that 
		\begin{equation*}
			\int_{B_1(0)}v_m(x)g_{m}(z_m + r_m x)\d x =0
		\end{equation*}	
		and, due to \eqref{eq:construction-sequence}, we know that 
		\begin{equation*}
			\norm{v_m}_{L^\infty(B_1(0))}\ge 1/2
		\end{equation*}
		for all $m\in \N$. 
		
		Next, we need to prove a growth bound on our blow up sequence $v_m$.\smallskip 
		
		\textit{Claim A.} We find a constant $C$ such that 
		\begin{equation*}
			\norm{v_m}_{L^\infty(B_R(0))}\le C R^{s+\alpha}
		\end{equation*}
		for any $m\in \N$ and any $1\le R\le 1/r_m$. \smallskip
		
		The proof of this claim follows with the same arguments as in \cite[(6.18)]{RoWe24}, \cite[Step 2 in proof of Proposition 2.7.3]{FeRo24}, or \cite[(3.16)]{SvWe25}. Note that we need to employ \autoref{prop:hopf} in order to establish $\|g_{m}\|_{L^\infty(B_r(z_m))}\ge Cr^s$. \smallskip
		
		The function $v_m$ satisfies an equation. For this, we define the sets
		\begin{equation*}
			D_m\coloneq   r_m^{-1}\big(-\{z_m\}+\Omega_{m}\big).
		\end{equation*}
		Due to \autoref{lem:basic-properties-Gs}, we find an operator $\tilde{L}_m\in \Gs$ such that for $x\in D_m\cap B_{1/r_m}(0)$
		\begin{equation}
			\tilde{L}_m v_m(x)= \frac{r_{m}^{s-\alpha}}{\theta(r_m)}L_{m}\big(u_{m}-q_mg_{m}\big)(z_m+r_mx)= \frac{r_m^{s-\alpha}}{\theta(r_m)}f_{m}(z_m + r_m x) 
		\end{equation}
		in the sense of distributions. This is also true for weak solutions since any weak solution is also a distributional solution, see \cite[Corollary 2.6.11]{FeRo24}. After passing to yet another subsequence, these operators $\tilde{L}_m$, too, converge to a limit operator $\tilde{L}_\infty\in \Gs$. The functions $x\mapsto r_m^{s-\alpha}f_{m}(r_m x)/\theta(r_m) $ converge to zero locally uniformly since $\|f_{m}\|_{L^\infty(\Omega_{m} \cap B_1 ) }\le 1$, $\theta(r_m)\to \infty$ as $m\to \infty$, and $\alpha<s$.
		
		This observation allows us to take limits. By \autoref{th:Cs-regularity}, we know that $\|v_m\|_{C^s(B_R)}\le C(R)$. The stability result \cite[Proposition 2.2.36]{FeRo24} yields a limit $v_\infty$ and a vector $\theta\in S^{d-1}$ such that 
		\begin{equation*}
			\begin{aligned}
				\tilde{L}_\infty v_\infty &= 0 \text{ in }\{ x\cdot \theta>0 \},\\
				v_\infty &= 0 \text{ on }\{x\cdot \theta\le 0\}
			\end{aligned}
		\end{equation*}
		and, due to claim A, the bound $|v_\infty(x)|\le C(1+|x|)^{s+\iota}$ holds for all $x\in \R^d$. The Liouville result \autoref{prop:liouville-half-space} yields $v_\infty = \kappa_1 b_\theta(x\cdot \theta)$ where $b_\theta$ is the one-dimensional solution \autoref{lem:half-space-solution} adapted to the rotation $\theta$. Since $\|v_m\|_{L^\infty(B_1(0))}\ge 1/2$, the constant $\kappa_1$ is not zero. 
		
		Note that the sequence of functions $g_m(x)\coloneq  r_m^{-s}g(z_m+r_mx)$ will also have a limit. They satisfy $\|g_m\|_{L^\infty(B_{1/r_m}(0))}\le C$ and 
		\begin{equation*}
			\begin{aligned}
				\tilde{L}_m g_m&= 0\phantom{r_m^{-s}\1_{B_{2/r_m}(0)^c}} \text{ in }\Omega_m \cap B_{1/r_m},\\
				g_m&= r_m^{-s}\1_{B_{2/r_m}(0)^c}\phantom{0}\text{ on }(\Omega_m \cap B_{1/r_m})^c
			\end{aligned}
		\end{equation*}
		Clearly, $|g_m(x)|\le C(1+|x|)^s$ for all $x\in \R^d$ and, due to \autoref{th:Cs-regularity}, the uniform bound $\|g_m\|_{C^s(R)}\le C(R)$ holds for all $0<R<1/(2r_m)$. Again, the stability result \cite[Proposition 2.2.36]{FeRo24} yields a locally uniform limit $g_\infty$ that solves
		\begin{equation*}
			\begin{aligned}
				\tilde{L}_\infty g_\infty= 0 \text{ in }\{ x\cdot \theta>0 \},\quad
				g_\infty= 0\text{ on }\{ x\cdot \theta\le 0 \}
			\end{aligned}
		\end{equation*}
		and satisfies $|g_\infty(x)|\le C(1+|x|)^s$. By the Liouville result, see \autoref{prop:liouville-half-space}, the identity $g_\infty(x)= \kappa_2 b_\theta(x\cdot \theta)$ holds. By the Hopf lemma, \autoref{prop:hopf}, we find a positive constant $C$ such that $r_m^{-s}g(z_m+r_m|x|\eta_{z_m})\ge C |x|^s$ for all $x\in B_{1/r_m}(0)$. Here, $\eta_{z_m}$ is the inward normal at $z_m\in \partial\Omega$. This inequality directly implies that the constant $\kappa_2$ cannot be zero. \smallskip
		
		By the choice of $q_m$ and $v_m$, we know 
		\begin{equation*}
			0= \int_{B_1(0)} v_m(x)r_{m}^{-s}g_{m}(z_m + r_m x)\d x \to \kappa_1\kappa_2 \int_{B_1(0)} b_\theta(x)^2\d x\ne 0.
		\end{equation*}
		This is a contradiction and, thus, the result follows. 
	\end{proof}
	
	\begin{lemma}\label{lem:quotient-regularity}
		Let $L,\eps,s,\Omega,\omega,g,f,$ and $u$ be as in \autoref{th:expansion}. For any $\alpha\in (0,s)$
		\begin{equation*}
			[u/g]_{C^\alpha(\Omega\cap B_{1/2})}\le C \big( \|u\|_{L_{2s-\eps}^\infty(\R^d)}  +\norm{f}_{L^\infty( \Omega \cap B_1)} \big).
		\end{equation*}
	\end{lemma}
	The proof of this lemma is very close to \cite[Theorem 1.2]{RS16}.
	\begin{proof}
		Without loss of generality, we assume $\alpha>s-\eps$ and $ \|u\|_{L_{2s-\eps}^\infty(\R^d)}  +\norm{f}_{L^\infty( \Omega \cap B_1)} \le 1$. We fix $x_0\in \Omega\cap B_{1/2}$ and $r\coloneq  d_{\Omega}(x_0)/2$. Let $z_0$ be the projection of $x_0$ onto the boundary. Due to \autoref{th:expansion}, we find a constant $q=q(z_0)$ such that for all $x\in B_{r}(x_0)$
		\begin{equation}\label{eq:expansion-help}
			|u(x)-qg(x)|\le C |x-z_0|^{s+\alpha}.
		\end{equation}
		
		\textit{Claim A.} We claim that $[u-qg]_{C^{\alpha}(B_r(x_0))}\le C r^{s}$. 
		
		In order to prove this, we define the auxiliary function $v_r(x)= r^{-s}(u(x_0+rx)-qg(x_0+rx))$. Clearly, we find an operator $L_r\in \Gs$ such that $L_r v_r(x)= r^{2s}f(x_0+rx)$ for $x\in B_{3/2}(0)$. The estimate \autoref{th:expansion} with the knowledge $u\in L_{2s-\eps}^\infty(\R^d)$ and $g\in L^\infty(\R^d)$ yields
		\begin{equation}\label{eq:v-r-bound-infinity}
			\|v_r\|_{L^\infty(B_R)}\le C \begin{cases}
				r^{\alpha} R^{s+\alpha}&, rR<1,\\
				r^{s-\eps}R^{2s-\eps}&, rR\ge 1.
			\end{cases}
		\end{equation}		
		The interior regularity estimate \cite[Theorem 2.4.3]{FeRo24} yields
		\begin{equation*}
			[v_r]_{C^{\alpha}(B_{1})}\le C\Big( \norm{v_r}_{L^\infty_{s+\alpha}(\R^d)} +r^{s}\norm{f}_{L^\infty(B_{3r/2}(x_0))} \Big).
		\end{equation*}
		Note that $[v_r]_{C^{\alpha}(B_{1/2})}= r^{\alpha-s}[u-qg]_{C^{\alpha}(B_{r/2}(x_0))} $ and, using \eqref{eq:v-r-bound-infinity},
		\begin{equation*}
			\norm{v_r}_{L^\infty_{2s-\eps}(\R^d)}\le C\sup_{rR<1} r^{\alpha}\frac{R^{s+\alpha}}{(1+R)^{s+\alpha}}+ C \sup_{rR\ge 1} \frac{r^{s-\eps}R^{2s-\eps}}{(1+R)^{s+\alpha}} \le Cr^{\alpha}. 
		\end{equation*}
		This yields claim A. \smallskip
		
		The triangle inequality reveals for $x,y\in B_{r}(x_0)$
		\begin{align*}
			\left|\frac{u(x)}{g(x)}- \frac{u(y)}{g(y)}\right|\le \frac{|(u-qg)(x)-(u-qg)(y)|}{g(x)}+ |u(y)-qg(y)| \frac{|g(x)-g(y)|}{g(x)g(y)}.
		\end{align*}
		By claim A and \autoref{prop:hopf}, the first term in the previous estimate is bounded by $C|x-y|^{\alpha}$. The second term is easily estimated using \eqref{eq:expansion-help} and the $C^s$-regularity of $g$ from \autoref{th:Cs-regularity}. This yields the same estimate which proves for a constant independent of $r$ and $x_0$
		\begin{equation*}
			[u/g]_{C^\alpha(B_r(x_0))}\le C.
		\end{equation*}
		This completes a proof in connection with a chain argument, see \cite[Proposition 1.1]{RS14}. 
	\end{proof}
	\begin{proof}[{Proof of \autoref{th:boundary-harnack}}]
		We write $U_1=u_1/g$, $U_2=u_2/g$. Due to \autoref{th:boundary-estimate}, \autoref{prop:hopf}, and \autoref{lem:quotient-regularity}, we estimate for $x,y\in B_{1/2}\cap \overline{\Omega}$
		\begin{align*}
			\left|\frac{u_1(x)}{u_2(x)}- \frac{u_1(y)}{u_2(y)}\right|\le \frac{|U_1(x)-U_1(y)|}{|U_2(x)|} + \frac{|U_1(y)| |U_2(x)-U_2(y)|}{|U_2(x)U_2(y)|}\le C |x-y|^{\alpha}.
		\end{align*}
	\end{proof}


\end{document}